\documentclass{amsart}

\usepackage{amsmath}
\usepackage{amsthm}
\usepackage{amssymb}
\usepackage{amsfonts}
\usepackage{bm}
\usepackage[shortlabels]{enumitem}
\usepackage[bookmarks=true,hyperindex,pdftex,colorlinks,citecolor=red, linkcolor=cyan]{hyperref}
\usepackage{centernot} 
\usepackage{marginnote} 
\usepackage{bbm}
\usepackage[all]{xy}
\usepackage{tikz}
\usetikzlibrary{arrows}
\usetikzlibrary{shapes,decorations}
\usetikzlibrary{positioning}
\usepackage{tikz-cd} 
\usepackage[normalem]{ulem}
\usepackage{enumitem}

\DeclareMathOperator{\lspan}{span}                          
\DeclareMathOperator{\supp}{supp}                           
\DeclareMathOperator{\diam}{diam}                           
\DeclareMathOperator{\Lip}{Lip}                             
\DeclareMathOperator{\lip}{lip}                             

\newcommand{\N}{\mathbb{N}}             
\newcommand{\R}{\mathbb{R}}             
\newcommand{\C}{\mathbb{C}}             

\newcommand{\abs}[1]{\left|{#1}\right|}                     

\newcommand{\set}[1]{\left\{{#1}\right\}}                   
\newcommand{\norm}[1]{\left\|{#1}\right\|}                  
\newcommand{\ball}[1]{B_{{#1}}}                             
\newcommand{\duality}[1]{\left<{#1}\right>}                 

\newcommand{\lipfree}[1]{\mathcal{F}({#1})}                 
\newcommand{\F}{\mathcal{F}}                                
\newcommand{\Free}{\mathcal{F}}                             
\newcommand{\lipnorm}[1]{\norm{#1}_L}                       

\newcommand{\restricted}{\mathord{\upharpoonright}}

\def\<{\langle}
\def\>{\rangle}

\newcommand{\indicator}[1]{{\mathbf 1}_{{#1}}}
\renewcommand{\hat}{\widehat}

\theoremstyle{plain}
\newtheorem{theorem}{Theorem}[section]
\newtheorem{lemma}[theorem]{Lemma}
\newtheorem{corollary}[theorem]{Corollary}
\newtheorem{proposition}[theorem]{Proposition}

\newtheorem{fact}[theorem]{Fact}

\theoremstyle{definition}
\newtheorem*{definition*}{Definition}
\newtheorem{definition}[theorem]{Definition}
\newtheorem{example}[theorem]{Example}
\newtheorem{question}{Question}

\theoremstyle{remark}
\newtheorem{remark}[theorem]{Remark}

\begin{document}
\title{Injectivity of Lipschitz operators}

\author[Luis C. Garc\'ia-Lirola]{Luis C. Garc\'ia-Lirola}
\address[L. Garc\'ia-Lirola]{Departamento de Matem\'aticas, Universidad de Zaragoza, 50009, Zaragoza, Spain}
\email{luiscarlos@unizar.es}

\author[C. Petitjean]{Colin Petitjean}
\address[C. Petitjean]{Univ Gustave Eiffel, Univ Paris Est Creteil, CNRS, LAMA UMR8050 F-77447 Marne-la-Vallée, France}
\email{colin.petitjean@univ-eiffel.fr}

\author[A. Proch\'azka]{Anton\'in Proch\'azka}
\address[A. Proch\'azka]{Laboratoire de Math\'ematiques de Besan\c con,
Universit\'e Bourgogne Franche-Comt\'e, Universit\'e de Franche-Comt\'e,
CNRS UMR-6623,
16, route de Gray,
25030 Besan\c con Cedex, France}
\email{antonin.prochazka@univ-fcomte.fr}


\begin{abstract}
Any Lipschitz map $f\colon M \to N$ between metric spaces can be ``linearised'' in such a way that it becomes a bounded linear operator $\widehat{f}\colon \F(M) \to \F(N)$ between the Lipschitz-free spaces over $M$ and $N$. The purpose of this note is to explore the connections between the injectivity of $f$ and the injectivity of $\widehat{f}$. While it is obvious that if $\widehat{f}$ is injective then so is $f$, the converse is less clear. Indeed, we pin down some cases where this implication does not hold but we also prove that, 
for some classes of  metric spaces $M$, any injective Lipschitz map $f\colon M \to N$ (for any $N$) admits an injective linearisation. Along our way, we study how Lipschitz maps carry the support of elements in free spaces and also we provide stronger conditions on $f$ which ensure that $\widehat{f}$ is injective.
\end{abstract}



\maketitle

\section{Introduction}

For a metric space $(M, d)$, the \textit{Lipschitz-free space} (also known as Arens\textendash{}Eells space or transportation cost space) $\F(M)$ is a Banach space which is built around $M$ in such a way that $M$ is isometric to a (linearly dense) subset $\delta(M)$ of $\F(M)$, and Lipschitz maps from $\delta(M)$ into any Banach space $X$ uniquely extend to bounded linear
operators from $\F(M)$ into $X$ (see Section~\ref{Prelim:Free} for a more detailed definition).
For metric space-valued maps on $M$ the linearisation procedure takes the following form.
For every Lipschitz map $f \colon M \to N$, there exists a continuous linear map $\widehat{f} \colon \F(M) \to \F(N)$ such that its operator norm is equal to the best Lipschitz constant of $f$, and moreover the following diagram commutes: 
	$$\xymatrix{
		M \ar[r]^f \ar[d]_{\delta_{M}}  & N \ar[d]^{\delta_{N}} \\
		\F(M) \ar[r]_{\widehat{f}} & \F(N)
	}$$   
\smallskip

A recent program, which motivated quite many specialists in the field, consists in trying to characterise (linear) properties of $\F(M)$ in terms of (metric) properties of~$M$. 
Yet, while there is a number of papers dealing with this scheme (see  \cite{AGPP21, APP2019, Ambrosio, DKP16, Godard, GLPRZ, Ostrak21, OO20, PRZ18} just to name a few), 
less papers focus on how the properties of Lipschitz maps $f$ and their linearisations $\widehat{f}$ are related. 
This is precisely the purpose of this note as we will focus on injectivity. While it is well known and rather easy to see that $f\colon M \to N$ is bi-Lipschitz if and only if $\widehat{f}$ is an isomorphic embedding, the question whether the injectivity of a Lipschitz map $f\colon M \to N$ implies the injectivity of the linearisation $\widehat{f}$ has not been dealt with; and is more delicate as we aim to show in the present note.
Let us remark right away, that the converse implication trivially holds. Indeed, using the same notation as in the commutative diagram above, $f$ can be naturally identified with $\widehat{f}\restricted_{\delta_M(M)}$.
Also, since simple examples can be produced when $M$ is not complete, see Section~\ref{Prelim:Free}, we restrict our attention to complete metric spaces $M$.
\smallskip

In this context, our main results are the following.
For some classes of metric spaces $M$, we prove that for every metric space $N$, every Lipschitz injection $f\colon M \to N$ has an injective linearization $\widehat{f} \colon \F(M) \to \F(N)$. We call the spaces satisfying the latter property \emph{Lip-lin injective spaces}. This is the case of, e.g., compact spaces having null 1-dimensional Hausdorff measure ($\mathcal H^1(M)=0$). In fact, we prove that if $M$ is compact and totally disconnected, then $M$ is Lip-lin injective if and only if the space $s_0(M)$ of locally constant Lipschitz functions from $M$ to $\R$ separates points of $M$ uniformly. That is, there exists a constant $C >0$ such that for every distinct points $x \neq y$ in $M$, one can find a $C$-Lipschitz map $f \in s_0(M)$ in such a way that $|f(x) -f(y)| = d(x,y)$. 
Further we prove that uniformly discrete metric spaces are Lip-lin injective while this is not the case for every discrete metric space (see Example~\ref{e:CountableDiscrete}).
Also, if $M$ is not purely 1-unrectifiable (p1u), i.e. there exists $A \subset \R$ of positive Lebesgue measure that bi-Lipschitz embeds into $M$, then $M$ is not Lip-lin injective.
This allows us to provide examples, on the one hand, of a Lip-lin injective compact $M$ such that $\mathcal H^1(M)>0$, and on the other hand, of a p1u, compact, totally disconnected  $M$ which is not Lip-lin injective.

As a preliminary work for the above results, we investigate sufficient conditions on $f$ which guarantee that $\widehat{f}$ is injective. This is done by analysing how Lipschitz operators carry the supports of elements in free spaces. It turns out that, when $M$ is bounded, $\widehat{f}$ is injective if and only if $\supp(\hat{f}(\mu))=\overline{f(\supp(\mu))}$ for every $\mu \in \F(M)$.

Finally, we give a complete solution to the related question of when $\widehat{f}^{**}$ is injective: this happens if and only if $f$ is a bi-Lipschitz embedding. We conclude the paper with a few observations about surjectivity of Lipschitz operators. 
\medskip

\subsection{Preliminaries on Lipschitz-free spaces}~
\label{Prelim:Free}
There are several ways to build Lipschitz-free spaces (see \cite{LipschitzBook, Weaver2}). Of course, the constructions are equivalent in the sense that they give birth to isometrically isomorphic Banach spaces. 
Here we will 
use spaces of Lipschitz functions.
\smallskip
 
Let $M$ be a \textit{pointed} metric space, that is a metric space equipped with a distinguished point denoted by $0$ (the choice of the base point is not relevant). If $N$ is another metric space then $\Lip_0(M,N)$ stands for the set of Lipschitz maps $f\colon M \to N$ such that $f(0_M) = 0_N$. We will also write $\Lip(f)$ for the best Lipschitz constant of $f$. When $(N,d_N) = (\R , |\cdot|)$, we will simply write $\Lip_0(M)$ instead of $\Lip_0(M,\R)$. Note that $\Lip_0(M)$ is actually a vector space, and when equipped with the norm
\[\displaystyle
\|f\|_L := \Lip(f) =  \sup_{x \neq y \in M} \frac{|f(x)-f(y)|}{d(x,y)},\] 
it naturally becomes a Banach space. 
\smallskip

Next, for $x\in M$, we consider the linear map $\delta(x) \colon \Lip_0(M) \to \R$ given by $\langle f,\delta(x) \rangle = f(x)$. It is readily seen that $\delta(x)$ is continuous with 
$$\|\delta(x)\|_{\Lip_0(M)^*} = d(x,0).$$ 
The \textit{Lipschitz-free space over $M$} is then defined as the closed subspace of $\Lip_0(M)^*$ generated by these evaluation functionals, that is, 
\[\F(M) := \overline{ \lspan}^{\| \cdot  \|}\left \{ \delta(x) \, : \, x \in M  \right \} \subset \Lip_0(M)^*.\]
The Lipschitz-free space over $M$ is characterised (up to isometric isomorphism) by the following \textit{``universal extension property''}: any Banach space-valued Lipschitz map $f\colon M \to X$ vanishing at $0$ can be extended in a unique way to a continuous linear map $\bar{f} \colon \lipfree{M} \to X$ whose operator norm is equal to the best Lipschitz constant of $f$. As a consequence (taking $X=\R$), $\F(M)$ is an isometric predual of $\Lip_0(M)$, and the corresponding weak$^\ast$ topology on $\ball{\Lip_0(M)}$ coincides with the topology of pointwise convergence.  Another easy consequence is the linearisation property which was already discussed at the beginning of the introduction.

\begin{proposition} \label{diagramfree}
	If $f \in \Lip_0(M,N)$, then there exists a unique bounded linear operator $\widehat{f} \colon \F(M) \to \F(N)$ with $\|\widehat{f}\|=\Lip(f)$ and $\delta_N \circ f  = \widehat{f} \circ \delta_M$.
\end{proposition}

Let us continue with a few more basic but important facts. For any subset $M'\subset M$ containing $0$, $\lipfree{M'}$ may be canonically identified with the closed subspace of $\lipfree{M}$ generated by the evaluation functionals on points of $M'$.
A fundamental tool, introduced in \cite{AP20} for bounded spaces and in \cite{APPP20} in the general case, is the \emph{support} of elements in Lipschitz-free spaces. Given $\mu\in \mathcal F(M)$, its \emph{support} $\supp\mu$ is the intersection of all closed subsets $K\subset M$ so that $\mu\in\mathcal F(K\cup\set{0})$.
The so-called Intersection Theorem (see~\cite{AP20,APPP20}) implies that $\mu \in \mathcal F(\supp(\mu)\cup \set{0})$.
Recall that $0$ is never an isolated point of $\supp(\mu)$ and that $\supp 0 =\emptyset$. Also, a point $p \in M$ lies in the support of $\mu$ if and only if for every neighbourhood $U$ of $p$ there
exists a function $f \in  \Lip_0
(M)$ whose support is contained in $U$ and such that
$\<f, \mu\> \neq 0$ (\cite[Proposition~2.7]{APPP20}). We will use these facts without any further reference.
Finally, the following well-known example see e.g. \cite[Example~3.11]{Weaver2}) will be crucial for us.

\begin{example}
\label{exam:IsomR}
$\F(\mathbb R) \equiv L^1(\R)$. Indeed, the linear map $\Phi : \F(\R) \to L^1(\R)$ such that 
$$ \forall x\in \mathbb R,
\quad \Phi(\delta(x)) = \left\{
\begin{array}{cc}
	\indicator{[0,x]} & \text{if } x \geq 0 \\
	-\indicator{[x,0]} & \text{if } x < 0,
\end{array}
\right.
 $$
is a surjective isometry.
\end{example}

\subsection{The Lipschitz map \texorpdfstring{$f$}{f} versus the Lipschitz operator \texorpdfstring{$\widehat{f}$}{f hat}}
\label{Prelim:LipOperator}
~

Even if the connections between $f \colon M \to N$ and $\widehat{f} \colon \F(M) \to \F(N)$ are not much explored yet, some links are well known and have been used repeatedly in the literature. 
For instance, the three first assertions below can be found in \cite{GoKa_2003} while the last one is \cite[Proposition~2.1]{ACP20}:
\begin{itemize}
	\item $f$ is bi-Lipschitz if and only if $\widehat{f}$ is an isomorphic embedding; 
	\item $f$ is a Lipschitz isomorphism (bi-Lipschitz and surjective) if and only if $\widehat{f}$ is a linear isomorphism;
	\item $f$ is a Lipschitz retraction if and only if $\widehat{f}$ is a linear projection;
	\item $f$ has dense range if and only if $\widehat{f}$ has dense range.
\end{itemize}
These assertions should be compared with \cite[Proposition~2.25]{Weaver2} where similar statements are proved for the composition operator $C_f \colon g \in \Lip_0(N) \mapsto g \circ f \in \Lip_0(M)$. Notice that, through the isometry $\F(M)^* \equiv \Lip_0(M)$, the composition operator $C_f$ is naturally identified with
the adjoint operator of $\widehat{f} \colon \F(M) \to \F(N)$. Indeed, for every $g \in \Lip_0(N)$ and $x\in M$, we have
\begin{align*}
	\< \big(\widehat{f}\;\big)^*(g) , \delta(x) \> = \< g , \widehat{f}(\delta(x)) \> = \< g , \delta(f(x)) \> = g \circ f (x) = \<C_f(g) , \delta(x) \>.
\end{align*}
Finally, Lipschitz operators which are compact are characterized in \cite{ACP21} in terms of metric properties of $f$ (see also \cite[Theorem~1.2]{Vargas1}). This program has also been addressed in other contexts. For instance,   each linear operator between Banach spaces corresponds to a lattice homeomorphism 
between the corresponding \emph{free Banach lattices}, in that case injectivity and surjectivity are always preserved as it is shown in \cite{OTTT}.   

\smallskip

Going back to the main subject  of the paper,
if $f$ is not injective, then there exist $x \neq y \in M$ with $f(x) = f(y)$, which implies that $\widehat{f}(\delta(x) - \delta(y)) = 0$. That is, $\widehat{f}$ is not injective. Furthermore, if $M$ is a metric space that is not complete, then $\F(M)$ is a Banach space that is linearly isometric to the Lipschitz-free space over the completion of $M$. This allows to provide ``trivial examples'' of injective Lipschitz maps with non-injective linearisation, such as the following one.

\begin{example}	
Let $M$ be non-complete metric space. Let $f\colon M \to N$ be a Lipschitz injection \emph{onto} some complete space $N$ such that $f(0)=0$. Let $(x_n) \subset M$ be a Cauchy sequence that does not converge in $M$. By completeness of $\Free(M)$, resp. of $N$, it is clear that $\mu:=\lim \delta(x_n) \in \Free(M)\setminus \delta(M)$ exists, resp. $\lim f(x_n) \in N$ exists. Let $x \in M$ be such that $f(x)=\lim f(x_n)$. 
By uniqueness of $\lim \hat{f}(\delta(x_n))$ we have that $\mu-\delta(x)\in \ker \hat{f}$. A concrete example of such situation is $M=[0,2\pi)$, $N=\set{z \in \C:\abs{z}=1}$ (where $0_N=1 \in \C$) and $f(x)=e^{ix}$. \end{example}

In light of the above example, from now on we will always tacitly assume that the considered metric spaces are complete.  
\medskip

\noindent \textbf{Notation.}  
Let us briefly describe the notation that will be used throughout this paper. For a Banach space $X$, we will write $B_X$ for its closed unit ball and $S_X$ for its unit sphere.
As usual, $X^*$ denotes the topological dual of $X$ and $\duality{x^\ast,x}$ will stand for the evaluation of $x^* \in X^*$ at $x \in X$. We will write $w = \sigma(X,X^*)$ for the weak topology in $X$ and $w^* = \sigma(X^*,X)$ for the weak$^*$ topology in $X^*$.

The letters $M$ and $N$ will always stand for \emph{complete} pointed metric spaces
with metric $d$ and base point $0$. 
The choice of the base point will be irrelevant to our results since, as is well known, free spaces over the same metric space but with different base points are isometrically isomorphic. 
Further, $B(p, r)$ will stand for the closed ball of radius $r$
around $p \in M$ and $\diam(A)=\sup\set{d(x,y)\colon x,y\in A}$ for the diameter of $A \subset M$.
Finally, $\Lip(M)$ stands for the space of all Lipschitz functions from $M$ to $\R$, and the Lebesgue measure on $\R$ will be denoted by $\lambda$.

\section{Preservation of supports and injectivity}
\label{section-PS}
Let us start with an easy observation. 

\begin{proposition}\label{preservsupplemma}
	Let $f \in \Lip_0(M,N)$. Then, for any $\mu \in \F(M)$,
	\[\supp\big( \widehat{f}(\mu) \big) \subset \overline{f\left(\supp(\mu)\right)}.\]
\end{proposition}

\begin{proof}
	Let $K$ be a closed subset of $M$.
	It is clear from the definitions that 
	\begin{equation}\label{preservsupp}
	\overline{\widehat{f}(\F(K\cup\set{0}))} = \F\big(\overline{f(K)}\cup\set{0}\big).
	\end{equation}
	Now let $\mu \in \F(M)$. 
	Since $\mu \in \F(\supp(\mu)\cup\set{0})$, we have $\widehat{f}(\mu) \in \widehat{f}\big(  \F(\supp(\mu) \cup\set{0}) \big)$. 
	Equality \eqref{preservsupp} implies that $\widehat{f}(\mu) \in \F\big( \overline{f\left(\supp(\mu)\right)}\cup\set{0} \big)$, which means that
	$\supp\big( \widehat{f}(\mu) \big) \subset \overline{f\left(\supp(\mu)\right)}$.
\end{proof}

Observe that the inclusion in Proposition~\ref{preservsupplemma} is strict whenever $\widehat{f}$ is non-injective. Indeed, if $\mu \neq 0 \in \F(M)$ is such that $\widehat{f}(\mu) = 0$, then $\supp \widehat{f}(\mu) = \emptyset$ while  $f(\supp{\mu}) \neq \emptyset$. This motivates the next definition. 

\begin{definition}
\label{def:preservation-support}
We say that a Lipschitz function $f \in \Lip_0(M,N)$ \emph{preserves the support of $\mu \in \F(M)$} if 
	\[\supp(\hat{f}(\mu))=\overline{f(\supp(\mu))}.\] 
If $f$ preserves the support of every $\mu \in \F(M)$, then we say that \emph{$f$ preserves supports}. 
\end{definition}

So, the observation before Definition~\ref{def:preservation-support} may be reformulated: if $f$ preserves supports, then $\widehat{f}$ is injective.  More precisely, we have the following proposition. 

\begin{proposition} \label{PSimpliesINJ}
Let $f\in \Lip_0(M,N)$. The following are equivalent:
	\begin{enumerate}
		\item[$(i)$] $f$ preserves supports.
		\item[$(ii)$] $f$ is injective and for any $\mu, \nu\in \mathcal F(M)$ with $\supp(\mu)\subset\supp(\nu)$, we have $\supp(\hat{f}(\mu))\subset\supp(\hat{f}(\nu))$.
	\end{enumerate}
	In any case, it follows that $\hat{f}$ is injective.
\end{proposition}

\begin{proof}
	$(i) \implies (ii)$. We already explained that $\widehat{f}$ is injective whenever $f$ preserves supports. So, we only have to prove the second part of the statement. Assume that $\supp(\mu)\subset\supp(\nu)$. Then one has
	\[\supp(\hat{f}(\mu))=\overline{f(\supp(\mu))}\subset\overline{f(\supp(\nu))} = \supp(\hat{f}(\nu)).\]
	
	$(ii) \implies (i)$. 
		Let us assume that $f$ does not preserve supports and that $f$ is injective. 
		Then there exists $\mu\in \Free(M)$ and $x\in \supp(\mu)$ such that $f(x) \notin \supp(\hat{f}(\mu))$.
		We claim that $x$ can be chosen so that $x \neq 0_M$. Indeed, assume that $0_M \in \supp(\mu)$ and $0_N=f(0_M) \notin \supp(\hat{f}(\mu))$. Then $0_M$ is not isolated in $\supp(\mu)$, so there exists a sequence $(x_n)_n \subset \supp(\mu)$ such that $x_n \to 0_M$ while $x_n \neq 0_M$ for every $n \in \N$. Since $f$ is continuous, we have $f(x_n) \to 0_N$. However, $0_N \notin \supp(\hat{f}(\mu))$ and $ \supp(\hat{f}(\mu))$ is closed, so there exists $N_0 \in \N$ such that $f(x_n) \notin \supp(\hat{f}(\mu))$ whenever $n \geq N_0$. 
		Now any $x_n$ with $n \geq N_0$ satisfies $x_n \in \supp(\mu)$ and $f(x_n) \notin \supp(\hat{f}(\mu))$, which were the required properties.
	
		So, let us fix $x\in \supp(\mu) \setminus \{0\}$ with $f(x) \notin \supp(\hat{f}(\mu))$. Since $f$ is injective, we have $f(x) \neq 0$. Therefore
		\[\set{x} = \supp(\delta(x)) \subset \supp(\mu) \quad \text{but} \quad \supp(\hat{f}(\delta(x)))=\set{f(x)} \not\subset \supp(\hat{f}(\mu)).\]
		Notice that $f(x) \neq 0$ was important for the last equality to hold. Indeed, by convention, the support of $0$ is the empty set.
	
	    Finally, the last statement has already been proved before Definition~\ref{def:preservation-support}.  
\end{proof}

Now a very natural question is whether every injective Lipschitz operator \linebreak $\widehat{f} \colon  \F(M) \to \F(N)$ is such that $f \colon M \to N$ preserves supports. Our next goal is to answer this question positively, in the case when the domain space $M$ is bounded. Before going into the details of the proof, let us provide some auxiliary remarks. 
\medskip

First, as we already mentioned, the adjoint operator of $\widehat{f} \colon \F(M) \to \F(N)$ can be naturally identified with a composition operator $C_f \colon g \in \Lip_0(N) \mapsto g \circ f \in \Lip_0(M)$. 
Since a bounded operator $T\colon X\to Y$ is injective if and only if $\overline{T^*(Y^*)}^{w^*}=X^*$ (see \cite[Exercise~2.44~(i)]{FHHMZ}) we immediately get the following.

\begin{fact}\label{lemma:separating} 
Let $f\in \Lip_0(M,N)$. Then $\hat{f}$ is injective if and only if $C_f(\Lip_0(N))$ is weak*-dense in $\Lip_0(M)$ (that is, $C_f(\Lip_0(N))$ is separating for $\mathcal F(M)$).
\end{fact}

Next, we will also need pointwise multiplication operators on Lipschitz spaces, and their pre-adjoints. Let $\omega\in\Lip(M)$ and let $K\subset M$ contain the base point and the support of $\omega$. For $f\in\Lip_0(K)$, let $M_\omega(f)$ be the function given by
\begin{equation}
\label{eq:T_h}
M_\omega(f)(x)=\begin{cases}
f(x)\omega(x) & \text{if } x\in K \\
0 & \text{if } x\notin K.
\end{cases} \,
\end{equation}
It is proved in \cite[Lemma~2.3]{APPP20} that if $\omega$ has bounded support then $M_\omega$ defines a weak$^*$-to-weak$^*$ continuous linear operator from $\Lip_0(K)$ into $\Lip_0(M)$, and $\norm{M_\omega}\leq\norm{\omega}_\infty+\displaystyle\sup_{x\in\supp(\omega)}d(0,x)\lipnorm{\omega}$. We will denote by $W_\omega \colon \F(M) \to \F(K)$ the pre-adjoint operator of  $M_\omega$. 
In fact, we will use a multiplication operator for a very particular Lipschitz map $\omega$ which will call ``$r$-plateau''.
\begin{definition}
 Let $x \in M$ and $r >0$. We will say that $\omega \in \Lip(M)$ is a \emph{$r$-plateau at $x$} if
\begin{itemize}
    \item $\omega(M)\subset [0,1]$,
    \item $\omega(B(x,r))=\set{1}$, and 
    \item $\omega(M\setminus B(x,2r))=\set{0}$.
\end{itemize}
\end{definition}
Such a map always exists. Indeed, we may define $w \colon B(x,r) \cup  (M\setminus B(x,2r)) \to \R$ by $w\equiv 1$ on $B(x ,r)$, $w \equiv 0$ on $M\setminus B(x,2r)$. Notice that $\Lip(w) \leq \frac{1}{r}$. Then, thanks to McShane--Whitney extension’s theorem (see e.g. \cite[Theorem~1.33]{Weaver2}), $w$ can be extended to a Lipschitz map $\tilde{w} \colon M \to \R$ with the same Lipschitz constant. Finally we let $\omega$ be given by $\omega(z) = \max(0 , \min(\tilde{w}(z),1))$. It is a routine exercise to check that $\omega$ satisfies the required properties with moreover $\Lip(\omega) \leq \frac{1}{r}$. 

The next result shows that, under a technical assumption that will be frequently satisfied, the preservation of supports is equivalent to the injectivity of the linearization.
Also, this result should be compared with Proposition~\ref{prop:biLipsupp}.

\begin{theorem} \label{theo:NRps}
Let $f\in \Lip_0(M,N)$ be  such that $\widehat{f}$ is injective and let $x \in M$. Assume that the following ``non-returning at $f(x)$'' condition is satisfied:
\begin{center}
    $(\mathcal{NR}_x)$: There exist $r,\rho>0$ such that $f(M)\cap B(f(x),\rho)\subset f(B(x,r))$ (equivalently $f^{-1}(B(f(x),\rho))\subset B(x,r)$).
\end{center} 
Then $f(x)\in\supp(\hat{f}(\gamma))$ whenever $x\in \supp(\gamma)$.\\
In particular, if $(\mathcal{NR}_x)$ holds for every $x\in M\setminus\set{0}$, then $\hat{f}$ is injective if and only if $f$ preserves supports.
\end{theorem}

\begin{proof}
We will argue by contradiction. 
Suppose $\gamma \in \F(M)$ is such that $x\in \supp(\gamma)$ but $f(x) \notin \supp(\widehat{f}(\gamma))$. 
We may assume that $\rho$ is small enough so that $B(f(x),\rho) \cap \supp(\widehat{f}(\gamma))=\emptyset$.
Take $0<r'<\min\set{r,\rho/(2\norm{f}_L)}$, note that $f(B(x,r'))\subset B(f(x),\rho/2)$. By injectivity of $f$ and condition ($\mathcal{NR}_x$), we also have $f(M\setminus B(x,r))\subset N \setminus B(f(x),\rho)$.
Let $\omega \in \Lip_0(N)$ be a $\rho/2$-plateau at $f(x)$. 
We claim that $M_{\omega\circ f}:\Lip_0(M) \to \Lip_0(M)$ is a bounded operator.
Indeed, let $g\in \Lip_0(M)$ and $a,b \in M$. 
If $a,b \in M\setminus B(x,r)$, then $\omega(f(a))=\omega(f(b))=0$ and 
$|M_{\omega \circ f}(g)(a) - M_{\omega \circ f}(g)(b)| = \abs{\omega(f(a))g(a)-\omega(f(b))g(b)} =0$.
Assume now without loss of generality that $a \in B(x,r)$. 
Then 
\[
\begin{aligned}
\abs{\omega(f(a))g(a)-\omega(f(b))g(b)}&\leq \abs{\omega(f(a))-\omega(f(b))}\abs{g(a)}+\abs{\omega(f(b))}\abs{g(a)-g(b)}\\
&\leq \norm{\omega\circ f}_L d(a,b) \norm{g}_Ld(0,a)+1\cdot\norm{g}_L d(a,b)\\
&\leq C \norm{g}_L d(a,b)
\end{aligned}
\]
for some suitable constant $C>0$
since $d(0,a)\leq d(0,x)+r$.
Now, using a standard argument involving the Banach-Dieudonn\'e theorem (see \cite[Lemma~2.3]{APPP20}), we obtain that $M_{\omega\circ f}$ admits a pre-adjoint $W$. 
Moreover it is easily checked that $M_{\omega\circ f}(\tilde{h}\circ f)= M_{\omega}(\tilde{h})\circ f$ for every $\tilde{h}\in \Lip(N)$. 
We further have that $\omega \circ f \equiv 1$ on $B\big(x,r'\big)$.
For an arbitrary $g\in \Lip_0(M)$ such that $\supp(g) \subset B\big(x,r')$, we appeal to Fact~\ref{lemma:separating} to get a net $(h_\alpha) \subset \Lip_0(N)$ such that $h_\alpha\circ f \to g$ weakly$^*$. We have that 
    \begin{align*}
        \< g , \gamma \> = \< M_{\omega \circ f}(g) , \gamma \> = \lim\limits_{\alpha} \< M_{\omega \circ f}(h_\alpha \circ f) , \gamma \> = \lim\limits_{\alpha} \< M_{\omega }(h_\alpha) \circ f , \gamma \> .
    \end{align*}
    We conclude by noticing that $M_\omega(h_\alpha) \in \Lip_0(N)$ is such that $\supp(M_\omega(h_\alpha)) \subset B(f(x) , \rho)$, which implies that $\lim\limits_{\alpha} \< M_{\omega }(h_\alpha) \circ f , \gamma \> = \lim\limits_{\alpha} \< M_{\omega }(h_\alpha), \widehat{f}(\gamma) \> = 0$.
    Since $g$ was arbitrary, this shows that $x \notin \supp(\gamma)$.
\end{proof}

We are now in good position to prove that an injective Lipschitz operator preserves the support of every element with bounded support.

\begin{corollary} \label{cor:injEQps}
 Let $f \in \Lip_0(M,N)$ be such that $\hat{f}$ is injective.
 Then $f$ preserves the support of every $\gamma \in \Free(M)$ such that $\supp(\gamma)$ is bounded.\\
 In particular, if $M$ is bounded then $\hat{f}$ is injective if and only if $f$ preserves supports.
\end{corollary}

\begin{proof}
Let $\gamma\in \Free(M)$ be an element with bounded support. Then, the map $g=f\restricted_{\supp(\gamma)\cup \set{0}}$ satisfies the non-returning condition at $f(x)$ for every $x\in \supp(\gamma)$, so 
\[\supp(\widehat{f}(\gamma))=\supp(\widehat{g}(\gamma))=\overline{g(\supp(\gamma))}=\overline{f(\supp(\gamma))}\]
by Theorem~\ref{theo:NRps}.
\end{proof}

We do not know if the last statement in Corollary~\ref{cor:injEQps} holds when $M$ is unbounded. Still, we obtain it holds for real-valued functions defined on a connected locally connected space.

\begin{lemma}\label{lemma:connectedNRpointwise}
 Assume that $M$ is connected  and let $x\in M$ admit a neighborhood basis made of connected sets. Let $f\in \Lip_0(M,\R)$ be injective. Then for every $r>0$ there exists $\rho>0$ such that $f(M) \cap B(f(x),\rho) \subset f(B(x,r))$.
\end{lemma}

\begin{proof}
Let $x\in M$ and $r>0$. Take a connected neighbourhood $U\subset B(x,r)$ of $x$. Then $I:=f(U)$ is an interval in $\mathbb R$.
If $f(x)$ lies in the interior of $I$, then we may find $\rho>0$ with $B(f(x),\rho)\subset I$ and we are done. Otherwise, $f(x)$ is an extreme point of $I$, let's assume for instance that $f(x)=\min I$. Now we claim that $f(x)=\min f(M)$. 
Indeed, consider $V=f^{-1}((-\infty, f(x)))$ and assume $V\neq \emptyset$. Since $V$ is open and since $M$ is connected, $V$  is not closed, and so there is a sequence $y_n\in V$ with $y_n\to y\in M\setminus V$. Then $f(y_n)<f(x)\leq f(y)$ and so $f(x)=f(y)$, which yields $y=x$ by injectivity. Then $y_n\in U$ eventually, so $f(y_n)\geq f(x)$, a contradiction. Now that we know that $f(x)=\min f(M)$, any $\rho>0$ with $[f(x),f(x)+\rho)\subset I$ will do the work.    
\end{proof}

As a direct consequence of Lemma~\ref{lemma:connectedNRpointwise} and Theorem~\ref{theo:NRps} we get the following. 

\begin{corollary}\label{cor:connected}  
Assume that $M$ is connected and locally connected and let $f \in \Lip_0(M,\R)$. 
Then $\hat{f}$ is injective if and only if $f$ preserves supports. 
\end{corollary}

We conclude the section with the next lemma which we use in the sequel.

\begin{lemma}\label{lemma:compositionsupp} Let $f\in \Lip_0(M,N)$ and $g\in \Lip_0(N,L)$. 
	\begin{enumerate}
		\item[(a)] If $f$ and $g$ preserve supports, then $g\circ f$ preserves supports.
		\item[(b)] If $g$ is closed and injective and $g\circ f$ preserves supports, then $f$ preserves supports.
		\item[(c)] If $f$ is closed, $\hat{f}$ is onto and $g\circ f$ preserves supports, then $g$ preserves supports. 
	\end{enumerate}
\end{lemma}

\begin{proof}
	(a) Let $\mu\in \mathcal F(M)$. Then 
	\[(g\circ f)(\supp\mu)= g(f(\supp\mu))\subset g(\supp(\hat{f}(\mu)))\subset \supp(\hat{g}(\hat{f}(\mu)))= \supp(\hat{g\circ f}(\mu)).\]
	
 (b) Let $\mu\in \mathcal F(M)$. By hypothesis, 
	\begin{align*}(g\circ f)(\supp\mu)\subset  \supp(\widehat{g\circ f}(\mu))= \supp(\hat{g} \circ\hat{f}(\mu))\subset \overline{g(\supp \hat{f}(\mu))}= g(\supp \hat{f}(\mu)).
	\end{align*}
	Since $g$ is injective, it follows that $f(\supp(\mu))\subset \supp\hat{f}(\mu)$, and taking closures yields the conclusion.
	\medskip
	
 (c) Let $\gamma\in \mathcal F(N)$ and take $\mu\in \mathcal F(M)$ with $\hat{f}(\mu)=\gamma$. Then $ g(\supp(\gamma))= g(\supp(\hat{f}(\mu))$ and
	\[ g(\supp(\hat{f}(\mu))\subset g\big(\overline{f(\supp\mu)}\big)=g(f(\supp\mu))\subset \supp(\widehat{g\circ f}(\mu))=\supp(\hat{g}(\gamma)).\]
\end{proof}

\begin{remark}
Both the statement and the proof of Lemma~\ref{lemma:compositionsupp} may be improved when $M$ is bounded. Indeed, 
since a Lipschitz map preserves supports if and only if it has injective linearization, we can prove only the corresponding statements for this latter property. Therefore (a) becomes trivial. Assertion (b) readily follows from the fact that $\widehat{g\circ f}=\hat{g}\circ \hat{f}$ is injective only if $\hat{f}$ is injective. Notice we do not need $g$ to be closed, and not even injective for ``$\implies$''. 
Again, since $\widehat{g\circ f}=\hat{g}\circ \hat{f}$ and since $\hat{f}$ is onto, (c) is easy, and still no closedness assumption is needed.
\end{remark}

\section{Sufficient conditions for injectivity}
\label{section3}

In this section, we will provide some metric conditions on $f$ which ensure that $\widehat{f}$ is injective. 

Recall that if $X$ is a Banach space, then we say that a subspace $S\subset X^*$ is \emph{norming} if there exists $C\geq 1$ such that, for every $x \in X$, $\|x\| \leq C \sup_{x^* \in B_S} |x^*(x)|$. Of course, it is clear that if $S$ is norming then $S$ is separating for $X$. In particular,  thanks to Fact~\ref{lemma:separating}, we obtain:

\begin{fact}
Let $f\in \Lip_0(M,N)$. If $C_f(\Lip_0(N))$ is norming for $\F(M)$ then $\hat{f}$ is injective.
\end{fact}

Let us point out that it follows from \cite[Proposition~3.4]{Kalton04} that a subspace $S$ of $\Lip_0(M)$ which is also a sublattice, such as $C_f(\Lip_0(N))$,  is norming if and only if it \emph{separates the points of $M$ uniformly}, meaning that \[\exists C \geq 1, \; \forall x \neq y \in M, \; \exists f \in CB_S, \quad |f(x) - f(y)| = d(x,y).  \]

\subsection{The case of bi-Lipschitz maps}

Let us recall that $f \colon M \to N$ is bi-Lipschitz if there exist $a,b > 0$ such that 
\[\forall x,y \in M,  \quad a\ d(x,y) \leq d(f(x), f(y)) \leq b \ d(x,y).\]
The next result is already known. 

\begin{proposition}\label{fchapinjbilip} Let $f\in \Lip_0(M,N)$. The following are equivalent:
	\begin{enumerate}[$(i)$]
		\item $f$ is bi-Lipschitz.
		\item $\widehat{f}$ is injective with closed range. 
		\item $C_f$ is onto. 
	\end{enumerate}
	In any case, $C_f(\Lip_0(N))$ is norming for $\F(M)$  and $f$ preserves supports. 
\end{proposition}

\begin{proof}
The equivalence $(i) \iff (iii)$ is contained \cite[Proposition~2.25]{Weaver2}. Next, $(ii)\iff (iii)$ is a general standard fact: an operator $T$ is an into isomorphism if and only if $T^*$ is onto (see Exercise~2.49 in \cite{FHHMZ}
for instance). 
To conclude, if $C_f$ is onto then it is clear that $C_f(\Lip_0(N))$ is norming for $\F(M)$, while $f$ preserves supports thanks to Theorem~\ref{theo:NRps}.

\end{proof}

However, there are support-preserving Lipschitz functions, even with  $C_f(\Lip_0(N))$ norming, which are not bi-Lipschitz. 

\begin{example}
Consider $M=N=[0,+\infty)$ and $f\colon M\to N$ given by $f(x)=x$ if $x\leq 1$ and $f(x)=\sqrt{x}$ if $x\geq 1$. Clearly $f$ is not bi-Lipschitz. We claim that $C_f(\Lip_0([0,+\infty))$ separates the points of $M$ uniformly. Indeed, given $x,y\in M$ with $x<y$, we have that $f^{-1}|_{[0,y]}$ is Lipschitz. Thus the function $g$ given by $g(t)=f^{-1}(t)$ if $0\leq t\leq y$ and $g(t)=f^{-1}(y)$ otherwise is Lipschitz,  and 
	\[(g\circ f)(t)=\begin{cases}t &\text{if } f(t)\leq y\\ f^{-1}(y)& \text{otherwise}.\end{cases}\] In particular $g \circ f$ is $1$-Lipschitz and satisfies $|g(f(x))-g(f(y))|=|x-y|$. Finally, note that $f$ satisfies  ($\mathcal{NR}_x$) for every $x$ and so $f$ preserves supports by Theorem~\ref{theo:NRps} or by our future Proposition~\ref{prop:biLipsupp}.
\end{example}

A similar counterexample can be given for a discrete metric space $M$. 

\begin{example}
Let $M=\N \cup\{0\}$ equipped with the metric $d_M$ verifying $d_M(0,n) = 1$ and $d_M(n,m) = d_M(n,0)+d_M(m,0)$ for every $n,m \in M$. Similarly we let $N=\N \cup\{0\}$ equipped with the metric $d_N$ such that $d_N(0,n) = 1/2^n$ and $d_N(n,m) = d_N(n,0)+d_N(m,0)$ for every $n,m \in N$. Then $Id\colon M \to N$ is clearly not bi-Lipschitz. However, it is readily seen that $C_{Id}(\Lip_0(N))$ separates points of $M$ uniformly. 
\end{example}

\subsection{The case of locally bi-Lipschitz maps}

\begin{proposition}
\label{prop:biLipsupp}
Let $f\in \Lip_0(M,N)$ be an injective map and let $x\in M$. 
Assume there are $r,\rho>0$ such that $f\restricted_{B(x,r)}$ is bi-Lipschitz and $f(M)\cap B(f(x),\rho)\subset f(B(x,r))$.
Then $f(x)\in\supp(\hat{f}(\gamma))$ whenever $x\in \supp(\gamma)$.
\end{proposition}

Notice that in particular, the second hypothesis is satisfied if $f(x)$ is isolated in $f(M)$. Notice also that assuming only that $f \in \Lip_0(M,N)$ is injective and locally bi-Lipschitz is not enough to conclude that $\hat{f}$ is injective; see Example~\ref{e:CountableDiscrete}.

\begin{proof} 
	We may and we do assume that $\overline{f(M)}=N$. Suppose that $\gamma \in \F(M)$ and $x \in M$ are such that $x\in \supp(\gamma)$ but $f(x) \not\in \supp(\hat{f}(\gamma))$.
	By assumption, there exists $r,\rho>0$ such that $f|_{B(x,r)}$ is bi-Lipschitz and $f(M)\cap B(f(x),\rho)\subset f(B(x,r))$. We may assume that $\rho$ is small enough so that  $B(f(x),\rho)\cap \supp(\hat{f}(\gamma))=\emptyset$. 
	Take $0<r'<\min\{r, \rho/\norm{f}_L\}$, note that $f(B(x,r'))\subset B(f(x), \rho)$. Since $x\in \supp(\gamma)$, there is $\varphi\in \Lip(M)$ such that $\supp(\varphi)\subset B(x,r')$ and $\langle\varphi,\gamma\rangle\neq 0$. Note that the function $g=\varphi\circ f^{-1}\colon f(M)\to\mathbb R$ is Lipschitz.  Indeed,  
	\begin{itemize}
		\item If $p, q\in f(B(x,r))$ then $|g(p)-g(q)|\leq \norm{\varphi}_L\norm{f^{-1}|_{f(B(x,r))}}_L d(p,q)$.
		\item If $p,q\in f(M)\setminus f(B(x,r'))$ then $f^{-1}(p), f^{-1}(q)\notin B(x, r')$ (since $f$ is injective) and so $g(p)=g(q)=0$. 
		\item If $p\in f(B(x,r'))$ and $q\in f(M)\setminus f(B(x,r))$. Then $d(q,f(x))\geq \rho$ (otherwise, $q\in f(M)\cap B(f(x),\rho))\subset f(B(x,r))$, a contradiction) and $g(q)=0$. Also, $d(p,f(x))\leq r' \norm{f}_L$. Thus, $d(p,q)\geq \rho- r'\norm{f}_L=:\alpha>0$. We have
		\[ |g(p)-g(q)|=|g(p)|\leq \norm{\varphi}_\infty \leq \frac{\norm{\varphi}_\infty}{\alpha}d(p,q).\]
	\end{itemize}
	Now, we can extend $g$ uniquely to a Lipschitz function $\tilde{g}$ on $N=\overline{f(M)}$. Since $g|_{f(M)\setminus B(f(x),\rho)}=0$, we also have $\tilde{g}|_{N\setminus B(f(x),\rho)}=0$, and so $\supp{\tilde{g}}\subset B(f(x),\rho)$. Thus $\langle \tilde{g}, \hat{f}(\gamma)\rangle=0$. 
	In addition, $\varphi=\tilde{g}\circ f$. Indeed, let $x \in M$, then $\tilde{g}(f(x))=g(f(x))=\varphi\circ f^{-1}(f(x))=\varphi(x)$. Thus, $\langle \tilde{g}, \hat{f}(\gamma)\rangle=\langle \tilde{g}\circ f, \gamma\rangle= \langle \varphi, f\rangle \neq 0$, a contradiction.
\end{proof}

\begin{remark}
It is clear that if the assumptions of Proposition~\ref{prop:biLipsupp} are satisfied for every $x \in M\setminus \{0\}$, then $f$ preserves supports and so $\widehat{f}$ is injective. 
We do not know if this implies the stronger property ``$C_f(\Lip_0(N))$ is norming for $\F(M)$''. 
It does, of course, when $M$ is compact, since then $f$ is bi-Lipschitz.
On the other hand, $C_f(\Lip_0(N))$ norming does not imply that the assumptions of Proposition~\ref{prop:biLipsupp} are satisfied for some $x$. 
A simple counterexample is the Lipschitz map $f = Id\colon([0,1],\abs{\cdot}^{1/2}) \to ([0,1],\abs{\cdot})$. It is straightforward that $f$ is not bi-Lipschitz in a neighborhood of any point, therefore the assumptions of Proposition~\ref{prop:biLipsupp} are not satisfied.
The fact that $C_f(\Lip_0(N))$ separates points uniformly follows from a more general result; see Subsection~\ref{SubsectionSPU}.
\end{remark}

\begin{corollary}\label{cor:conlocconps} 
Assume that $M$ is connected and locally connected, and $f\in \Lip_0(M, \mathbb R)$ is an injective function which is locally bi-Lipschitz on $M\setminus\{0\}$.
Then $f$ preserves supports.
\end{corollary}
\begin{proof}
Apply Lemma~\ref{lemma:connectedNRpointwise} and Proposition~\ref{prop:biLipsupp}. 
\end{proof}
A simple example where the above applies is $f \colon x \in [0,1] \mapsto x^2 \in [0,1]$.
We shall provide more applications of Proposition~\ref{prop:biLipsupp}.

\begin{corollary}\label{lemma:closureisolated} 
Let $M$ and $N$ be any metric spaces. Let $f\in \Lip_0(M, N)$ be an injective closed map, $\gamma\in \mathcal F(M)$ and $x\in \supp(\gamma)$. 
If $x$ is in the closure of isolated points of $\supp(\gamma)$, then $f(x)\in \supp(\hat{f}(\gamma))$.
\end{corollary} 

\begin{proof}
Suppose first that $x$ is isolated in $\supp(\gamma)$. Let $0<r<d(x, \supp(\gamma)\setminus\{x\})$ and $0<\rho < d\big(f(x),f(\supp(\gamma)\setminus\{x\}) \big)$ (the fact that $f$ is closed is used here), and consider the map $f\restricted_{\supp (\gamma)\cup\set{0}}\colon \supp(\gamma)\cup \set{0}\to f(\supp(\gamma)\cup \set{0})$. Then a direct application Proposition~\ref{prop:biLipsupp} implies that $f(x)\in \supp(\hat{f}(\gamma))$.

Otherwise, take a sequence $(x_n) \subset \supp(\gamma)$ of isolated points of $\supp(\gamma)$ such that $x_n \to x$.
We have that $\supp(\hat{f}(\gamma))\ni f(x_n)\to f(x)$, thus $f(x) \in \supp(\hat{f}(\gamma))$ as the support is closed. 
\end{proof}

Notice that it follows rather directly from Corollary~\ref{lemma:closureisolated} that if $M$ is compact and countable, then $f\colon M\to N$ injective implies $\hat{f}$ injective for every $N$ and every $f\in \Lip_0(M,N)$. This statement will be improved in Corollary~\ref{c:compact-liplin-metric}.
Finally, we also obtain a ``reduction to bounded metric spaces'' kind of result. 

\begin{corollary}\label{l:BddArrivalReduction}
	Let $f\in \Lip_0(M,N)$. Consider the metric on either $M$ or $N$ given by $\rho(x,y)=\min\{1,d(x,y)\}$. If we write $Id_M \colon x \in (M,d)\mapsto x \in (M,\rho)$, $Id_N\colon  x \in (N,d)\mapsto x \in  (N,\rho)$ and $f_{\rho} \colon x \in (M, \rho) \mapsto f(x) \in (N,\rho)$, then:
	\begin{enumerate}[$(a)$]
	    \item $Id_M$, $Id_N$ and $f_{\rho}$ are Lipschitz;
	    \item $f$ preserves supports if and only if $Id_N\circ f$ preserves supports; 
	    \item $f$ preserves supports whenever $f_{\rho}$ preserves supports.
	\end{enumerate}
\end{corollary}
\begin{proof}
Assertion $(a)$ is clear. Next, note that $Id_N$ is a closed map which preserves supports thanks to Proposition~\ref{prop:biLipsupp}. So we may apply Lemma~\ref{lemma:compositionsupp} to obtain $(b)$. Finally, assume that $f_{\rho}$ preserves supports.  Since $Id_N \circ f = f_{\rho} \circ Id_M$, according to $(b)$, $f$ preserves supports if and only if $f_{\rho} \circ Id_M$ does so. 
Since $Id_M$ preserves supports, Lemma~\ref{lemma:compositionsupp}~$(a)$ implies that $f_\rho \circ Id_M$ preserves supports, and so $f$ preserves supports.  
\end{proof}

\subsection{Uniform separation of points}
\label{SubsectionSPU}

We will now provide some sufficient conditions on $f \colon M \to N$ which ensure that $C_f(\Lip_0(N))$ is norming for $\F(M)$. 

Before stating and proving the next proposition, let us introduce some terminology and notation. 
For $f:[0,\infty) \to \R$ and $g:[0,\infty)\to\R$ we define their \emph{inf-convolution} as
\[
f \square g (s)=\inf_{a+b=s} f(a)+g(b) \quad \mbox{ for all }s\geq 0.
\]
If $g(0)=0$ then $f\square g(s)\leq f(s)$ for all $s\geq 0$. 
If $C\geq 1$, we say that $f:[0,\infty)\to [0,\infty)$ is \emph{$C$-subadditive} if $f(a+b)\leq f(a)+Cf(b)$ for every $a,b\geq 0$.
It is easy to prove that if $f,g$ are $C$-subadditive then $f\square g$ is $C$-subadditive. 

\begin{proposition}\label{propSPU}
    Let $f\in \Lip_0(M,N)$ be an injective map. 
    Assume that there exist non-decreasing functions $\omega :[0,\infty) \to [0,\infty)$ such that 
    \begin{enumerate}[$(a)$]
        \item $\omega(0) = 0$ and $\omega$ is left continuous;
        \item $\exists C_1\geq 1$ such that $\omega$ is $C_1$-subadditive;
        \item $\exists C_2>0$ such that for every $x,y \in M$:    
            \begin{equation}\label{e:moduli}
            C_2 \omega(d(f(x),f(y)))  \leq d(x,y) \leq \omega(d(f(x),f(y))).
            \end{equation}
    \end{enumerate}
    Then $C_{f}(\Lip_0(N))$ separates the points of $M$ uniformly and $f$ preserves supports. In particular, $\hat{f}$ is injective.
\end{proposition}

\begin{proof}
    Fix $x \neq y \in M$. 
    For $n\in \N$ we define $\omega_n=\omega\square nId$.
    This is the largest $n$-Lipschitz function below $\omega$.
    Notice moreover that $\omega_n(0) = \omega(0) = 0$ and:
    \begin{itemize}
        \item[(i)] $\omega_n  \colon [0,+\infty) \to [0,+\infty)$ is non-decreasing;
        \item[(ii)] $\lim\limits_{n \to \infty} \omega_n(t) = \omega(t)$ for every $t \in [0 , \infty)$; 
        \item[(iii)] $|\omega_n(t_1) - \omega_n(t_2)| \leq C_1 \omega(|t_1-t_2|)$ for all $t_1,t_2 \in [0,+\infty)$.
    \end{itemize} 
    In order to prove (ii) when $t>0$, fix $\varepsilon>0$ and let $a< t$ be such that $\omega(a)>\omega(t)-\varepsilon$. 
    Take $n$ such that $n(t-a)\geq \omega(a)$. Then $\omega(b)+n(t-b)\geq \omega(a)>\omega(t)-\varepsilon$ for every $b\leq t$.
    
    In order to prove (iii) assume that $0\leq t_1<t_2$. 
    Then by $C_1$-subadditivity of $\omega_n$ and by the fact that $nId(0)=0$  we have
    \[
    \abs{\omega_n(t_1)-\omega_n(t_2)}=\omega_n(t_2)-\omega_n(t_1)\leq C_1 \omega_n(t_2-t_1)\leq  C_1 \omega(t_2-t_1).
    \]
        Next, for every $n \in \N$ we let $g_n \in \Lip_0(N)$ be the map given by
	\[\forall z \in N, \quad  g_n(z) = \omega_n(d(z,f(y))) - \omega_n(d(f(y),0)). \]
	For every $x_1, x_2 \in M$ we have by (iii)
	\begin{align*}
	    |g_n ( f(x_1) ) - g_n ( f(x_2) )| &\leq C_1 \omega\Big(\big|d( f(x_1) , f(y) ) - d( f(x_2) , f(y) ) \big| \Big) \\
	   &\leq  C_1  \omega( d( f(x_1) , f(x_2) ) ) \\
	   &\leq \frac{C_1}{C_2} d( x_1 , x_2 ).
	\end{align*}
	Hence $\|C_f(g_n)\|_{L} \leq \frac{C_1}{C_2}$. 
	Moreover, the point (ii) above implies that
	\[|g_n(f(x)) - g_n(f(y))| = \omega_n( d(f(x),f(y)) ) \underset{n \to \infty}{\longrightarrow} \omega( d( f(x) , f(y) ) ) \geq d(x,y).\]
	Thus $C_f(\Lip_0(N))$ separates the points of $M$ uniformly and so $\hat{f}$ is injective. 
	To conclude, note that the second inequality in \eqref{e:moduli} implies that $(\mathcal {NR}_x)$ holds for every $x\in M$, so $f$ preserves supports by Theorem~\ref{theo:NRps}.
\end{proof}

\begin{corollary}\label{Prop_Snow_SPU}
	Let $\alpha\in (0,1)$ and assume $(M,d)$ is bounded. 
	Then $Id \colon (M,d^\alpha) \to (M,d)$ is  a Lipschitz map such that $C_{Id}(\Lip_0(M, d))$ separates the points of $(M,d^\alpha)$ uniformly. In particular, $\widehat{Id}$ is injective.
\end{corollary} 

\begin{proof}
For the choice $\omega(t)=t^{\alpha}$ and $C_1=C_2=1$, all the requirements of Proposition~\ref{propSPU} are satisfied.
\end{proof}

\section{Metric spaces where injectivity is always preserved}

As the title of the section suggests, we will try to distinguish here the metric spaces $M$ such that $f\colon M\to N$ injective implies $\hat{f}$ injective for every $N$ and every $f\in \Lip_0(M,N)$ from those metric spaces which do not have this property.
For future reference, let us call such spaces \emph{Lip-lin injective spaces}. 

Let us give some basic properties of Lip-lin injective spaces. Obviously, Lip-lin injectivity is stable under Lipschitz equivalences, i.e. if $d$ and $\rho$ are Lipschitz equivalent metrics on $M$, then either both $(M,d)$ and $(M,\rho)$ are Lip-lin injective, or they both are not Lip-lin injective.
We will see later that this does not translate to an isomorphic property of $\Free(M)$.
Also, in the definition of Lip-lin injective we might consider only bounded metric spaces $N$.

\begin{lemma} \label{l:bounded-reduction}
	A metric space $M$ is Lip-lin injective if and only if for every bounded metric space $N$ and every $f\in \Lip_0(M,N)$, $\widehat{f}$ is injective whenever $f$ is injective.  
\end{lemma}

\begin{proof}
    One implication is trivial and the other one is proved with a similar trick as in Corollary~\ref{l:BddArrivalReduction}. Indeed, assume that $M$ is not Lip-lin injective. Then there exists a metric space $N$ and a Lipschitz map $f\in \Lip_0(M,N)$ such that $f$ is injective but $\widehat{f}$ is not injective. Let $\rho$ be the metric on $N$ given by $\rho(x,y)=\min\{1,d(x,y)\}$ and let $Id_N\colon x \in  (N,d)\mapsto x \in (N,\rho)$. Next, let $g := Id_N \circ f : M \to (N , \rho)$. Clearly $(N,\rho)$ is a bounded metric space and $g$ is injective since both $f$ and $Id_N$ are injective. Moreover, since $\widehat{f}$ is not injective and $\widehat{g} = \widehat{Id_N} \circ \widehat{f}$, we easily obtain that $\widehat{g}$ is not injective.
\end{proof}

Further, Lip-lin injectivity is hereditary in the following sense.

\begin{lemma}\label{l:heredity}
Let $M$ be a metric space. If $M$ is Lip-lin injective and $W$ bi-Lipschitz-embeds into $M$, then $W$ is Lip-lin injective.
\end{lemma}
\begin{proof}
Suppose that $W$ bi-Lipschitz-embeds into $M$ but $W$ is not Lip-lin injective. Therefore there exists a metric space $N$ and an injective $f\in \Lip_0(W, N)$ with non-injective $\hat{f}$. We will prove that $M$ cannot be Lip-lin injective. Without loss of generality, we may assume that $0\in W \subset M$ and,
using the Fr\'echet embedding, that $N \subset \ell_\infty(N)$ isometrically.
Let $\tilde{f}\in \Lip_0(M,\ell_\infty(N))$ be a Lipschitz extension of $f$. Let $\rho$ be the metric on $M$ given by $\rho(x,y) = \min\{1,d(x,y)\}$ and let $\varphi(x)=\rho(x,W)$.
Then $g\colon (M, d) \to \mathbb \ell_\infty(N) \times \Free(M,\rho)$ defined by $g(x):=(\tilde{f}(x),\varphi(x)\delta(x))$ is
 Lipschitz. Indeed, $\tilde{f}$ is Lipschitz and 
\begin{align*}
    \|\varphi(x) \delta(x) - \varphi(y) \delta(y)\|_{\F(M,\rho)} &\leq |\varphi(x)| \; \|\delta(x)-\delta(y)\| + \|\delta(y)\| \; |\varphi(x) - \varphi(y)| \\
    &\leq  \rho(x,y) + \rho(y,0) \rho(x,y) \\
    &\leq 2 d(x,y).
\end{align*}
Moreover, it is readily seen that $g$ is injective while $\ker \hat{f} \subseteq \ker \hat{g}$, which concludes the proof.
\end{proof}

\begin{remark}
Notice that the growth of the dimension of the target space is unavoidable. 
For instance, consider $M=[0,1]^2$ with the Euclidean distance.
While the space $[0,1]$ embeds isometrically into $M$ and there is an injective Lipschitz map $f:[0,1]\to [0,1]$ such that $\hat{f}$ is non-injective (see Example~\ref{exam:NotLipLin}), there is no continuous injective function $g:M \to \R$.
This remark also explains why in the definition of Lip-lin injective spaces we use the universal quantifier on the target space.
\end{remark}

Note also that whenever we have $f\colon M\to N$ Lipschitz and injective, one can consider the metric $\rho(x,y)=d(f(x),f(y))$ on $M$, satisfying that the identity map $Id\colon M\to (M, \rho)$ is Lipschitz and injective. Moreover, considering $g\colon (M,\rho)\to N$ given by $g(x)=f(x)$, we have that $g$ is an into isometry such that $g\circ Id =f$. It follows then that $\hat{f}$ is injective if and only if $\hat{Id}$ is injective. As an immediate consequence, we get that Lip-lin injective spaces can be characterized just by looking at identity maps:

\begin{proposition} A metric space $(M,d)$ is Lip-lin injective if and only if for every metric $\rho$ on $M$ such that $Id\colon (M,d)\to (M,\rho)$ is Lipschitz, we have that $\hat{Id}$ is injective.
\end{proposition}

\subsection{The compact case}

Now we will restrict our attention to the study of Lip-lin injectivity for compact $M$. A prominent role will be played by
the set
\[ s_0(M)=\{f\in \Lip_0(M): f\text{ is locally constant}\}.\]
We will also consider the space of locally flat Lipschitz functions (also known as little-Lipschitz functions), that is: 
\[ \lip_0(M):= \{f\in \Lip_0(M) : \forall \varepsilon>0 \ \exists \delta>0 : \sup_{0<d(x,y)<\delta} \frac{|f(x)-f(y)|}{d(x,y)}<\varepsilon\}.\]
It is readily seen that $s_0(M)$ is a vector sublattice of $\Lip_0(M)$ such that $s_0(M)\subset \lip_0(M)$.
The next lemma justifies our interest in $s_0(M)$; it says that locally constant functions on compact sets get always conserved in the image of composition operators with injective symbol. 

\begin{lemma}\label{l:LocConstConserved}
Let $M$ be a compact metric space and let $f\colon M \to N$ be Lipschitz and injective. 
Then $s_0(M) \subset C_f(\Lip_0(N))$.
\end{lemma}
\begin{proof}
Clearly, since $M$ is compact, $f\in s_0(M)$ 
if and only if it has only finitely many values. Therefore, if $\varphi \in s_0(M)$ then we can write $\varphi(M)=\set{a_1,\ldots,a_n}$.
Notice that the sets $\varphi^{-1}(a_i)$ cover $M$, are compact and pairwise disjoint. 
Therefore they are mutually at positive distance.
The same is true for $f(\varphi^{-1}(a_i))$: they cover $f(M)$, are compact and mutually at positive distance. 
Thus, if we define $g(x)=a_i$ for every $x \in f(\varphi^{-1}(a_i))$, we will have $g \in \Lip_0(f(M))$.
It is clear that $\varphi=\tilde{g}\circ f$ for any Lipschitz extension $\tilde{g}\in \Lip_0(N)$ of $g$ and we are done.
\end{proof}

With the help of the above lemma, one can see that, for $M$ compact, the weak$^*$-density of $s_0(M)$ implies that $M$ is Lip-lin injective. 
The next lemma shows in particular that, for those $M$ where $s_0(M)$ separates points of $M$, the converse is true even in the non-compact case.

\begin{lemma}\label{lemma:liplindenselattice}
Let $M$ be a Lip-lin injective metric space and $W\subset \Lip_0(M)$ be a vector lattice separating the points of $M$. Then $W$ is weak$^*$-dense in $\Lip_0(M)$.
\end{lemma}

\begin{proof}
Consider the distance in $M$ given by
\[ \rho(x,y):=\sup_{\varphi\in W\cap B_{\Lip_0(M,d)}} |\varphi(x)-\varphi(y)|.\]
Clearly, the identity $Id\colon(M,d)\to(M,\rho)$ is a 1-Lipschitz injective map. Note also that if $\varphi\in W$ then $\norm{\varphi\circ Id^{-1}}_{\Lip_0(M,\rho)}\leq\norm{\varphi}_{\Lip_0(M,d)}$. Setting $W'=\{\varphi\circ Id^{-1}: \varphi\in W\}$, it follows that $W'$ separates the points of $(M,\rho)$ uniformly. Since $W'$ is a subspace of $\Lip_0(M,\rho)$ which is also a lattice, we get from Proposition~3.4 in~\cite{Kalton04} that it is norming for $\mathcal F(M,\rho)$.

Now, assume that $W$ is not weak$^*$-dense. Then it does not separate the points of $\mathcal F(M)$, i.e. there exists $\gamma\in\mathcal F(M)\setminus\{0\}$ such that $\<\varphi,\gamma\>=0$ for all $\varphi \in W$. Then
\[\<\varphi\circ Id^{-1}, \hat{Id}(\gamma)\>= \<\varphi,\gamma\>=0 \]
whenever $\varphi\circ Id^{-1}\in W'$. Since $W'$ is norming, we get $\hat{Id}(\gamma)=0$. This contradicts the assumption that $M$ is Lip-lin injective. 
\end{proof}

Recall that a metric space $M$ is totally disconnected if the connected components in $M$ are the one-point sets. Also, we say that $M$ is totally separated if for every $x\neq y \in M$, there exists a clopen set $F$ such that $x\in F$ and $y\notin F$. It is readily seen that every totally separated metric space is totally disconnected. Moreover, if the metric space is compact then the converse is true (see page 20 in \cite{Disco}). 
The next theorem may be considered as a functional characterization of compact totally disconnected Lip-lin injective spaces.

\begin{theorem}\label{t:CharCptTotDiscoLLI} 
Let $M$ be a compact metric space. Then the following properties are equivalent.
\begin{itemize}
    \item[$(i)$] $s_0(M)$ is weak$^*$-dense in $\Lip_0(M)$,
    \item[$(ii)$] the set $\set{f \in \Lip_0(M):\lambda(f(M))=0}$ is weak$^*$ dense in $\Lip_0(M)$,
    \item[$(iii)$] $s_0(M)$ separates points of $M$ uniformly,
    \item[$(iv)$]  the set $\set{f \in \Lip_0(M):\lambda(f(M))=0}$ separates points of $M$ uniformly,
    \item[$(v)$] $M$ is Lip-lin injective and totally disconnected.
\end{itemize}
\end{theorem}
Notice that we cannot remove the assumption of compactness in this theorem (see  Example~\ref{e:CountableDiscrete}).
Indeed, since $M$ in this example is discrete we have $\Lip_0(M)=s_0(M)$.
Also, we do not know if $M$ Lip-lin injective implies that $M$ is totally disconnected.

\begin{proof}
$(i) \implies (v)$. Let $N$ and $f\in \Lip_0(M,N)$ injective be fixed. 
By Lemma~\ref{l:LocConstConserved} we have $s_0(M) \subset C_f(\Lip_0(N))$. Thus Fact~\ref{lemma:separating} yields the Lip-lin injectivity. 
It is clear that the weak$^*$-density of $s_0(M)$ implies that $M$ is totally disconnected.
Indeed, if there is a non-trivial connected component containing two points $x\neq y$, then $\duality{\varphi,\delta(x)-\delta(y)}=0$ for every $\varphi \in s_0(M)$ and so $s_0(M)$ is not weak$^*$-dense.
\smallskip

$(v) \implies (i)$. Since $M$ is totally disconnected and compact, it is totally separated and therefore $s_0(M)$ separates the points of $M$. Lemma~\ref{lemma:liplindenselattice} yields that $s_0(M)$ is weak*-dense in $\Lip_0(M)$.
\smallskip

$(iii) \implies (i)$. Indeed, since $s_0(M)$ is a vector lattice, Proposition~3.4 in~\cite{Kalton04} shows that if  $s_0(M)$ separates points of $M$ uniformly, then $s_0(M)$ is norming. In such case, of course, $s_0(M)$ is weak$^*$-dense in $\Lip_0(M)$.
\smallskip

$(i) \implies (iii)$. Since $s_0(M) \subset \lip_0(M)$, (i) implies that $\lip_0(M)$ separates points of $\Free(M)$. A standard application of the theorem of Petunin and Plichko~\cite{Petunin}  yields that $\lip_0(M)^*=\Free(M)$ (see \cite{Dalet} for more details).
Now (i) implies that $s_0(M)$ is weakly dense in $\lip_0(M)$. Hence Mazur's lemma gives that $s_0(M)$ is norm dense in $\lip_0(M)$.
Since for every $x\neq y \in M$ and every $\varepsilon>0$ there is $f\in \lip_0(M)$, $\norm{f}<1$ such that $\duality{f,\delta(x) - \delta(y)} >(1-\varepsilon)d(x,y)$, there is also $\varphi \in s_0(M)$ with the same properties.

The implications $(i) \implies (ii)$ and $(iii) \implies (iv)$ are trivial
and the proof of the converse implications 
follows immediately from Lemma~\ref{l:FVApproximationOfNV} below.
\end{proof}

\begin{lemma}\label{l:FVApproximationOfNV}
 Let $M$ be a metric space and let $f \in \Lip_0(M)$ such that $\lambda(\overline{f(M)})=0$. Then there is $(\varphi_n)\subset s_0(M)$ such that for each $n$ we have $\norm{\varphi_n}_L\leq \norm{f}_L$ and $\varphi_n \to f$ weakly$^*$.
\end{lemma}

\begin{proof}
It is enough to show that for any closed negligible $A \subset \R$ containing 0 there is a sequence $(\psi_n)$ of finitely-valued 1-Lipschitz functions $\psi_n\in s_0(A)$ such that $\psi_n \to Id\restricted_A$ pointwise.
Indeed, in this case $\varphi_n:=\psi_n\circ f$ satisfies the conclusion of the lemma as the weak$^*$ and the pointwise convergence coincide on bounded sets of Lipschitz functions.

The proof of the claim can be inferred from the proof of Theorem~4.3 in~\cite{AlPePr_2019}. For the reader's convenience we furnish the details.
We write $\R \setminus A=\bigcup I^+_j \cup \bigcup I^-_j$ where $I^\pm_j$ are pairwise disjoint open intervals such that $I^+_j \subset (0,+\infty)$ and $I^-_j\subset (-\infty,0)$.
For every $n \in \N$ we define $\psi_n(0)=0$ and
\[
\psi_n(x)=
\sum_{j\leq n, x\geq \sup I^+_j}\lambda(I^+_j) \;\; - \sum_{j\leq n, x\leq \inf I^-_j}\lambda (I^-_j).
\]
The required properties of $(\psi_n)$ are easy to check.
\end{proof}

In the following corollary we list some metric conditions which imply that $M$ is Lip-lin injective (and totally disconnected).

\begin{corollary}\label{c:compact-liplin-metric}
 Let $M$ be a compact metric space. Any of the following conditions implies that $M$ is Lip-lin injective. 
\begin{itemize}
\item[(a)] 
The one-dimensional Hausdorff measure of $M$ is $0$. 
\item[(b)] There exists $\rho>1$ such that for every $\varepsilon>0$, $M$ can be covered by finitely many  balls $B(x_i,r)$ of radius $r \leq \varepsilon$ such that the balls $B(x_i,\rho r)$ are pairwise disjoint.
\item[(b')] There exists $\rho>0$ such that for every $\varepsilon>0$, $M$ can be covered by finitely many closed sets $E_i$ such that $\sup_i \diam(E_i)=r\leq \varepsilon$ and the sets $[E_i]_{\rho r} := \{x \in M \; | \; d(x,E_i) \leq \rho r \}$ are pairwise disjoint.
\item[(c)] For every $x \in M$, $\lambda(\set{d(x,y):y\in M})=0$.
\item[(c')] There exists $C>0$ such that for every $x\neq y \in M$, there exists $\varphi \in CB_{\Lip_0(M)}$ satisfying $\varphi(x)-\varphi(y)\geq d(x,y)$ and $\lambda(\varphi(M))=0$.
\end{itemize}
\end{corollary}

\begin{proof}
The condition (c') is the same as condition (iv) in Theorem~\ref{t:CharCptTotDiscoLLI}.
We have (a) $\implies$ (c) $\implies$ (c') and (b) $\implies$ (b') $\implies$ (c').
This is clear for the first chain and for (b) $\implies$ (b'). 
In order to prove (b') $\implies$ (c') we let $(\varepsilon_n)_n \subset \R_+$ be decreasing to 0. 
Let us fix $n$ and let $E_1,\ldots, E_m$ and $0<r\leq\varepsilon_n$ correspond to $\varepsilon_n$. 
Choose arbitrary $x_i \in E_i$. 
We construct the retraction $r_n\colon M \to M$ by $r_n(x)=x_i$ if and only if $x \in  E_i$.
Let $x\in E_i$ and $y\in E_j$.
Then $d(x,y)\geq \rho r$ and
\[
\begin{aligned}
\frac{d(r_n(x),r_n(y))}{d(x,y)}&\leq \frac{d(x,y)+2r}{d(x,y)} 
\leq 1+\frac{2r}{\rho r}=1+\frac{2}{\rho}.
\end{aligned}
\]
Thus, we have $\norm{r_n}_{L}\leq 1+\frac{2}{\rho}$.
It follows that for every $\varphi \in \Lip_0(M)$ we have $\varphi \circ r_n \to \varphi$ pointwise and so $s_0(M)$ separates points of $M$ uniformly.
A fortiori, (c') is satisfied.
\end{proof}

\begin{remark}\label{r:CantorDust}
The above conditions (a), (b) and (c) have been chosen since they can be (relatively) easily verified. 
Conditions (b') and (c') are counterparts of (b) and (c), respectively, that are moreover invariant under Lipschitz isomorphisms. This can be easily checked.
On the other hand, we claim that neither (b) nor (c) are invariant under Lipschitz isomorphisms\footnote{Still, (b) is invariant under Lipschitz isomorphisms with sufficiently small distortion.}. This shows that (c') does not imply (c), and (b') does not imply (b). 

To see our claim for (c), consider the Cantor dust $D=C\times C$ 
(here $C$ stands for the middle-third Cantor set). 
If $D$ is equipped with the $\ell_1$-metric, then it is well known that $\set{\norm{x-0}_1:x \in D}=C+C=[0,2]$ which has Lebesgue measure 2. On the other hand, if $D$ is equipped with the $\ell_\infty$-metric then $D$ satisfies (c).

To see our claim for (b), consider the following modification of the Cantor dust. Let $M=\bigcap M_n$ where \[M_1=([0,\frac13] \cup [\frac23,1])\times ([0,\frac13] \cup [\frac23,1]) \cup \set{\frac12}\times \set{\frac16,\frac56} \cup \set{\frac16,\frac56}\times \set{\frac12}\] and $M_{n+1}$ is obtained from $M_n$ by replacing every square (product of intervals) by an appropriately scaled down copy of $M_1$.
It can be verified that when $M$ is equipped with the $\ell_1$-metric then (b) fails, while when it is equipped with the $\ell_\infty$-metric then (b) is satisfied.

Also, $(D,\norm{\cdot}_{1})$ shows that (b) does not imply (c).
That (a) does not imply (b') can be seen on the following example. Let $M=\set{0}\cup \set{\frac1n:n\in \N} \subset \R$ and let us assume that $M$ satisfies (b') with some $\rho$. Let $\varepsilon>0$, $r\leq \varepsilon$ and $E_1,\ldots,E_m$ as in (b'). Without loss of generality $0 \in E_1$. Let $n$ be minimal such that $\frac1n\in E_1$. Notice that $n\geq \frac1\varepsilon$. Then $d(0,\frac1n)\leq r$ and $d(\frac1n,\frac1{n-1})>\rho r$. So $n-1=\frac{d(0,\frac1n)}{d(\frac1n,\frac1{n-1})}\leq\frac1\rho$. Since $\varepsilon$ was arbitrary, this leads to a contradiction.

Finally, notice that (c') is satisfied for example by $\mathcal{H}^1-\sigma$-finite purely 1-unrectifiable metric spaces by a theorem of Choquet \cite{Choquet} (see also~\cite[page 3554]{AGPP21}) and the fact that $\lip_0(M)$ separates points uniformly in p1u spaces \cite[Lemma 3.4]{Bate_2020} (see also \cite[Theorem~A]{AGPP21}).
\end{remark}

\begin{remark} 
    The condition (b) comes from the paper by Godefroy and Ozawa~\cite[Proposition 6]{GodOza} where they prove that if $M$ satisfies (b), then $\Free(M)$ has the metric approximation property (MAP).
    Following that proof, one can see that (b') also implies the MAP.
    It has been pointed to us by the anonymous referee that if $M$ satisfies (b'), then $\F(M)$ has even a finite dimensional decomposition (FDD). Indeed, let $r_n : M \to M$ be the Lipschitz retractions defined in the proof of Corollary~\ref{c:compact-liplin-metric}. Then, similarly as in the proof of~\cite[Proposition~6]{GodOza}, $\widehat{r_n} : \F(M) \to \F(M)$ are finite rank projections converging strongly to the identity on $\F(M)$. In other words, $\F(M)$ has the $\pi$-property. Moreover, the subspace $\lip_0(M) \subset \Lip_0(M)$ of locally flat Lipschitz maps separates points uniformly (simply because $s_0(M) \subset \lip_0(M)$). Therefore $\F(M) = \lip_0(M)^*$ (this is again an application of the theorem of Petunin and Plichko~\cite{Petunin}; see also \cite{Dalet} for more details). Now Grothendieck’s theorem shows that $\F(M)$ has the MAP since it is a separable dual with the bounded approximation property (BAP). In conclusion, it follows from \cite[Theorem~6.4 (3)]{Casazza} that $\F(M)$ has a FDD.
\end{remark}

\subsection{The case of uniformly discrete metric spaces}

We are now going to show that every uniformly discrete metric space is Lip-lin injective. Let us recall that $M$ is uniformly discrete if there exists $\theta>0$ such that $d(x,y) > \theta$ whenever $x \neq y$. If $M$ is moreover bounded, then $\F(M)$ is readily seen to be isomorphic to $\ell_1(M \setminus \{0\})$ through the linear map $\delta(x) \mapsto e_{x}$ (see \cite[Proposition~4.4]{Kalton04}).

\begin{lemma}\label{l:ACseries} Let $M$ be any metric space. Assume that $\mu=\sum_{n=1}^\infty a_n\delta(x_n)=\sum_{n=1}^\infty b_n\delta(y_n)$ in $\mathcal F(M)$, where $(a_n), (b_n)\in \ell_1$, $a_n,b_n\neq 0$ for all $n$, and $0\neq x_n\neq x_m$, $0\neq y_n\neq y_m$ if $n\neq m$. Then there is a permutation $\sigma\colon \mathbb N\to\mathbb N$ such that $a_n=b_{\sigma(n)}$ and $x_n=y_{\sigma(n)}$ for all $n$.
\end{lemma}

In the language of Fremlin and Sersouri~\cite{FS}, this means that the family $\delta(M\setminus \set{0})$ is $\ell_1$-independent. We wish to highlight that this Lemma can be derived from \cite[Proposition~4.9]{AP21} in straightforward manner, but we choose to include a direct self-contained proof below.

\begin{proof}
	Consider $\mu=\sum_{n=1}^\infty a_n \delta(x_n)$. Given $x\in M$ and $\varepsilon>0$, we consider the Lipschitz function given by $f_{x,\varepsilon}(t)=\max\{1-\frac{d(t,x)}{\varepsilon},0\}$. If $x\notin\{x_n:n\in \mathbb N\}$ then
	\[|\langle f_{x,\varepsilon},\mu \rangle| \leq \sum_{d(x_n,x)<\varepsilon} |a_n| \stackrel{\varepsilon\to 0^+}{\longrightarrow} 0.\] 
	On the other hand, if $x=x_{n_0}$ for some $n_0\in \mathbb N$, then
	\[|\langle f_{x,\varepsilon}, \mu \rangle- a_{n_0}| =\Big|\sum_{n\neq n_0} a_n f_{x,\varepsilon}(x_n)\Big|\leq \sum_{\substack{ n\neq n_0\\
        d(x_n,x)<\varepsilon}} 
        |a_n|\stackrel{\varepsilon\to 0^+}{\longrightarrow} 0 \]
	That is, $\lim_{\varepsilon\to 0^+} \langle f_{x,\varepsilon}, \mu\rangle = a_{n_0}$ if $x=x_{n_0}$ for some $n_0$ and $0$ otherwise. The same argument yields $\lim_{\varepsilon\to 0^+} \langle f_{x,\varepsilon}, \mu\rangle = b_{n_0}$ if $x=y_{n_0}$ for some $n_0$ and $0$ otherwise, and the conclusion follows.
\end{proof}

\begin{corollary}\label{cor:UDsupp} 
Let $M$ be any metric space. Let $f \in \Lip_0(M,N)$ be injective and $\mu\in \mathcal F(M)$ such that $\mu=\sum_{n=1}^\infty a_n \delta(x_n)$ for some $(a_n)\in \ell_1$. Then $\supp{\hat{f}(\mu)}=\overline{f(\supp(\mu))}$.\\
In particular, if $\supp(\mu)$ is uniformly discrete and bounded then $f$ preserves the support of $\mu$. 
\end{corollary}

\begin{proof}
Let $\mu=\sum_{n=1}^\infty a_n \delta(x_n)$ for certain $x_n\in M$ and $a_n\neq 0$. By Lemma~\ref{l:ACseries}, we have $\supp(\mu)=\overline{\set{x_n}}$. Indeed, the inclusion ``$\subset$'' is trivial and the inclusion ``$\supset$'' follows from the proof of the case $x=x_{n_0}$ and the closedness of support. The conclusion follows. 
\end{proof}

\begin{corollary}\label{cor:UnifDisInj} 
If $M$ is a uniformly discrete metric space then $M$ is Lip-lin injective. 
\end{corollary}

\begin{proof} 
Assume that $M$ is uniformly discrete. Let $N$ be any metric space and $f \colon M \to N$ be any injective Lipschitz map vanishing at 0. We may and do assume that $\diam(N,d)\leq 1$
thanks to Lemma~\ref{l:bounded-reduction}. Also, exactly as in Corollary~\ref{l:BddArrivalReduction}, we consider the metric $\rho$ on $M$ given by $\rho(x,y)=\min\{1,d(x,y)\}$. Then clearly $(M , \rho)$ is uniformly discrete and bounded. Moreover we can factor $f$ through $(M,\rho)$ as $f=f_{\rho}\circ Id_M$ where $Id_M \colon x \in (M,d) \mapsto x \in (M,\rho)$ and $f_{\rho} \colon (M,\rho)\to N$ is defined by $f_{\rho}(x)=f(x)$. Now by Corollary~\ref{cor:UDsupp}, we know that $f_{\rho}$ preserves supports. Therefore, since $(N,d)=(N,\rho)$, Corollary~\ref{l:BddArrivalReduction}~$(c)$ implies that $f$ preserves supports as well.  
\end{proof}

\section{Counter-examples and the transfer method}
\label{section5}

The main goal of the current section is to provide examples of complete metric spaces $M$ which are not Lip-lin injective. That is, complete metric spaces $M$ for which there is a metric space $N$ and an injective Lipschitz map $f \colon M \to N$ such that $\widehat{f} \colon \F(M) \to \F(N)$ is not injective.
That such spaces exist follows for example from Theorem~\ref{t:CharCptTotDiscoLLI}.
Indeed, let $M$ be a compact totally disconnected subset of $\R$ with $\lambda(M)>0$.
Then by~\cite[Corollary~3.4]{Godard}, $\Free(M)$ contains an isometric copy of $L_1$. In particular, $\Free(M)$ is not a dual and, by the theorem of Petun{\={\i}}n and Pl{\={\i}}{\v{c}}ko
\cite{Petunin}, the space $\lip_0(M)$
does not separate points of $\Free(M)$ (again, see \cite{Dalet} for more details). A fortiori, $s_0(M)$ (which is a subspace of $\lip_0(M)$) does not separate points of $M$ and so $M$ is not Lip-lin injective.

In order to construct ``much smaller'' metric spaces which are not Lip-lin injective by means of a transfer method, we need a more tangible description of the above example.
For the sake of clarity, we will begin with a simple example, namely $M = [0,1]$, and then we will develop further the main idea to obtain examples of a different metric nature. 
We will need the following description of the linearization of $f$ in the particular case of subsets of $\R$.

\begin{lemma}\label{lemma:L1}
 Let $f\colon [0,1] \to [0,1]$ be an injective Lipschitz map with $f(0)=0$.
 Let $\Phi\colon \Free([0,1]) \to L_1[0,1]$ be the usual isometric isomorphism (see Example~\ref{exam:IsomR}).
 Then for every $\varphi \in L_1[0,1]$ we have $\Phi\circ \widehat{f}\circ \Phi^{-1}(\varphi)=\varphi \circ f^{-1}$.
\[ \begin{tikzcd}
\Free([0,1]) \arrow{r}{\widehat{f}}  & \Free({f([0,1])}) \arrow{d}{\Phi} \\
L_1[0,1] \arrow{u}{\Phi^{-1}} \arrow{r}{\cdot \circ f^{-1}}& L_1({f([0,1])})
\end{tikzcd}
\]
\end{lemma}

\begin{proof}
First of all, notice that our assumptions imply that $f$ is increasing.
Let us denote $T\colon L_1[0,1] \to L_1({f([0,1])})$ the composition operator given by $T\varphi = \varphi \circ f^{-1}$. This is a bounded operator. 
 Indeed, 
 \[
  \int_{f([0,1])} \abs{\varphi\circ f^{-1}(t)}dt=\int_{[0,1]}\abs{\varphi\circ f^{-1}(f(t))} \; |f'(t)|dt\leq \norm{\varphi}_1\norm{f}_L.
 \]
 Further, we have $\widehat{f}=\Phi^{-1}\circ T\circ \Phi$ (by  the uniqueness of $\widehat{f}$ it is enough to check only on evaluation functionals $\delta(x)$, $x \in [0,1]$).
 So applying $\Phi$ from the right and $\Phi^{-1}$ from the left we get the desired result. 
\end{proof}

From now on, let $C \subset [0,1]$ be the Smith-Volterra-Cantor set. That is, the space constructed similarly as the middle-third Cantor set, but, at the $n$-th step of the construction  we remove subintervals of width $1/4^n$ from the middle of each of the $2^{n-1}$ remaining intervals. Therefore $C$ is a closed and totally disconnected subset of $[0,1]$ such that $\lambda(C) \in (0,1)$, $\min C  = 0$ and $\max C  = 1$ (in fact, any subset having these properties would work). 

\begin{example} \label{exam:NotLipLin}
\textit{There exists an injective Lipschitz map $f \colon [0,1] \to [0,1]$ such that $\widehat{f} \colon \F([0,1]) \to \F([0,1])$ is not injective.}
\smallskip

The above statement can easily be derived from \cite{Weaver2}. Indeed, letting $\rho$ be the metric on $[0,1]$ given by $\rho(x,y)=\lambda([x,y]\setminus C)$, it clear that
$f=Id\colon ([0,1],|\cdot|)\to([0,1],\rho)$ is $1$-Lipschitz and injective. Now it is proved in \cite[Example 2.30]{Weaver2} that $C_f$ is an into isometry. In particular its pre-adjoint $\hat{f}$ is onto (see \cite[Exercise~2.49]{FHHMZ}). Then, $\hat{f}$ is not injective because otherwise it would be an isomorphism, which would mean that $f$ is bi-Lipschitz according to Proposition~\ref{fchapinjbilip}, and this is excluded.
Since we will need a concrete representation of an element in $\ker \hat{f} \setminus \set{0}$ later, we provide a different proof below.

\begin{proof}
We define $f\colon ([0,1],|\cdot|) \to ([0,1],|\cdot|)$ as 
\[f(x)=\lambda([0,x]\setminus C)=\int_0^x \indicator{[0,1]\setminus C} (t)dt.\]
It is clear that $f$ is $1$-Lipschitz and non-decreasing. Moreover $f(0)=0$, $f(1)=1-\lambda(C)>0$ and $f$ is injective. Indeed, if $x<y$ then there exist $a<b$ in $(x,y)$ such that $[a,b] \cap C=\emptyset$. 
Thus $f(y)-f(x)=\lambda([x,y]\setminus C)\geq b-a>0$. 
So $f$ is injective.
Finally, a simple integration by substitution gives
\[ \lambda(f(C)) = \int_{f(C)} 1 dt = \int_{C} f'(x)dx = \int_{C} \indicator{[0,1]\setminus C}(x)dx =0 .\]
Now, let $T$ be as in Lemma~\ref{lemma:L1} above for this particular function $f$.
Notice that $0\neq \indicator{C} \in L_1[0,1]$ but we have $T\indicator{C}=\indicator{C} \circ f^{-1}=\indicator{f(C)}=0 \in L_1[0,1]$.
By Lemma~\ref{lemma:L1} it follows that $0\neq \Phi^{-1}(\indicator{C}) \in \ker \widehat{f}$.
\end{proof}
\end{example}

Notice that the non-zero vector in the kernel of $\widehat{f}$ defined above, namely $\Phi^{-1}(\indicator{C})$, can be written in a more explicit way. Indeed, if we let $(x_{2^{n-1}+k} \,, \, y_{2^{n-1}+k})$, $k=0,\ldots, 2^{n-1}-1$, be the subintervals which we remove at the $n$-th step of the construction of $C$, then 
$$\Phi^{-1}(\indicator{C}) = \delta(1)-\sum_{n=1}^\infty (\delta(y_n)-\delta(x_n)).$$
Observe that $\Phi^{-1}(\indicator{C}) \in \F(C)$. 
We subsequently deduce that $\widehat{f\restricted_C}$ is not injective. 
\medskip

Paraphrasing the above construction, we readily obtain the next result.

\begin{proposition}  \label{prop:SubsetR}
If $A \subset \R$ is such that $\lambda(A)>0$ then $A$ is not Lip-lin injective.\\
In particular, if $A \subset \R$ is compact,
 then $A$ is Lip-lin injective if and only if \linebreak$\lambda(A) = 0$.
\end{proposition}

\begin{proof}
The proof follows the same lines as Example~\ref{exam:NotLipLin}. We will only underline the main arguments, details are left to the reader. Since $\lambda(A)>0$, it contains a subset $K$ that is compact, totally disconnected and $\lambda(K)>0$. We may assume that $0=\min K$ and let us denote $b=\max K$. We pose $f(x)=\frac{b}{\lambda([0,b]\setminus K)}\lambda([0,x]\setminus K)$ if $x\geq 0$ and $f(x)=x$ if $x<0$. Then we
prove similarly that $f$ is 1-Lipschitz, injective and moreover $\lambda(f(K)) = 0$. Finally, if $T$ is the operator given by Lemma~\ref{lemma:L1} for this particular $f$, then we observe that $T(\indicator{K}) = \indicator{f(K)} = 0$. Therefore $T$ is non-injective, and so is $\widehat{f}$.

The second part of the statement now readily follows from Corollary~\ref{c:compact-liplin-metric}.
\end{proof}

As a direct consequence of the fact that ``Lip-lin injective'' is a hereditary property, we obtain the following corollary. 

\begin{corollary}\label{c:LipLinInjectiveImpliesp1u}
 If $M$ is Lip-lin injective then $M$ is purely 1-unrectifiable.
\end{corollary}

\begin{proof}
Assume that $M$ is not purely 1-unrectifiable. Then there exists a closed $A\subset \R$ such that $\lambda(A)>0$ and $A$ embeds bi-Lipschitz into $M$. Now Proposition~\ref{prop:SubsetR} and Lemma~\ref{l:heredity} yield the conclusion. 
\end{proof}

We already witnessed that $M$ being compact and totally disconnected is not sufficient to be Lip-lin injective (simply take $M = C$).
The next result shows in particular that adding moreover the assumption ``$M$ is purely 1-unrectifiable'' does not change that fact. Indeed, the following proposition applies for instance to the case of the ``snowflake metric'' $|\cdot|^{\alpha}$ ($0<\alpha <1$) which turns subsets of $[0,1]$ into purely 1-unrectifiable metric spaces.

\begin{proposition}\label{p:p1uNotSufficient} Assume that $\rho$ is a metric in $[0,1]$ such that $Id\colon ([0,1],\rho) \to ([0,1],\abs{\cdot})$ 
is $L$-Lipschitz 
and
$Id^{-1}\colon ([0,1],\abs{\cdot})\to([0,1],\rho)$ is continuous.\\
Then there exist a totally disconnected set $A \subset [0,1]$ and an injective Lipschitz map $f\colon (A,\rho) \to ([0,1],\abs{\cdot})$ such that $\hat{f}$ is not injective.
\end{proposition}

\begin{proof} 
By compactness, $Id^{-1}$ is uniformly continuous.
We define:
$$\omega(t)=\sup\{\rho(x,y): |x-y|\leq t\}.$$ 
By  hypothesis, $t/L\leq \omega(t)\stackrel{t\to 0^+}{\to}0$. 
Take $t_n>0$ with $\omega(t_n)\leq 4^{-n}\min\{1,L^{-1}\}$ and consider the Smith-Volterra-Cantor set $C_{\rho}\subset [0,1]$ obtained by removing $2^{n-1}$ intervals $(x_{2^{n-1}+k},y_{2^{n-1}+k})$, $k=0,\ldots, 2^{n-1}-1$, of length $t_n$ at the $n$-th stage of the usual construction of a Cantor set. Note that $0<\lambda(C_{\rho})<1$.
\smallskip

Now exactly as in Example~\ref{exam:NotLipLin} or as in Proposition~\ref{prop:SubsetR}, the map $h\colon ([0,1],\abs{\cdot}) \to ([0,1],\abs{\cdot})$ defined by 
\[\forall x \in [0,1], \quad h(x)=\lambda([0,x]\setminus C_\rho)=\int_0^x \indicator{[0,1]\setminus C_\rho} (t)dt,\]
is $1$-Lipschitz and injective. Moreover 
$$\gamma := \delta(1)-\sum_{n=1}^\infty \delta(y_n)-\delta(x_n) \in \ker \widehat{h}.$$

Next, define $f\colon([0,1], \rho)\to ([0,1],|\cdot|)$ by $f=h\circ Id$.
Consider $\mu =\delta(1)-\sum_{n=1}^\infty \delta(y_n)-\delta(x_n) \in \F([0,1],\rho)$. 
Note first that this is a non-zero well-defined element in $\mathcal F([0,1],\rho)$ since 
\[ \sum_{n=1}^\infty \norm{\delta(y_n)-\delta(x_n)}_{\mathcal F([0,1], \rho)}=\sum_{n=1}^\infty \rho(x_n,y_n)\leq\sum_{n=1}^\infty 2^{n-1} \omega(t_n)\leq \sum_{n=1}^\infty \frac{2^{n-1}}{4^n}\frac{1}{L} =\frac{1}{2L}\]
and so 
$$\norm{\mu}\geq \norm{\delta(1)}-\frac{1}{2L}=\rho(1,0)-\frac{1}{2L} \geq \frac{1}{L} -\frac{1}{2L} = \frac{1}{2L} >0.$$ 
However, $\mu\in \ker(\hat{f})$ since $$\hat{h}(\hat{Id}(\mu))=\hat{h}(\gamma)=0.$$
Finally, notice that if we set $A:=C_{\rho}$, then $(A,\rho)$ is totally disconnected since $Id:([0,1], \rho) \to ([0,1], |\cdot|)$ is a homeomorphism. Also, $\mu \in \Free(A)$.
\end{proof}

Recall that a \textit{metric arc} is a metric space which is homeomorphic to the unit interval $[0,1]$. Also, we say that a metric space $M$ is \textit{bounded turning} if there exists a constant $C\geq 1$ such that every pair of points $(x,y) \in M^2$ is contained in a compact and connected set $S$ such that $\diam(S) \leq C d(x,y)$. In \cite{HerronMeyer12}, a family of metrics $(d_\Delta)_{\Delta \in \mathcal D}$ on $[0,1]$ is defined in such a way that every bounded turning metric 
arc is bi-Lipschitz equivalent to some curve $([0,1] , d_\Delta)$, $\Delta \in \mathcal D$.
We are grateful to Chris Gartland for bringing \cite{HerronMeyer12} to our attention and suggesting the next corollary.
\begin{corollary}
Let $M$ be a metric space. 
If $M$ contains a bounded turning metric arc, then $M$ is not Lip-lin injective. 
\end{corollary}
\begin{proof}
Let $M$ be a metric space and let $\Gamma \subset M$ be a bounded turning metric arc. 
On the one hand, thanks to \cite[Theorem~A]{HerronMeyer12}, $\Gamma$ is bi-Lipschitz equivalent to some curve $([0,1] , d_\Delta)$, $\Delta \in \mathcal D$. On the other hand, it is straightforward to check that $|\cdot| \leq d_{\Delta}$. Thus $([0,1] , d_\Delta)$ is not Lip-lin injective thanks to Proposition~\ref{p:p1uNotSufficient}. Therefore the same is true for $\Gamma$ and $M$ in view of Lemma~\ref{l:heredity}.
\end{proof}

We conclude the section with another two adaptations of Example~\ref{exam:NotLipLin}. The first example shows that one can not remove the boundedness assumption in the last statement of Proposition~\ref{prop:SubsetR}.

\begin{example}\label{e:CountableDiscreteR}
\textit{There exists a countable complete and discrete $M\subset \mathbb R$ and a Lipschitz and injective $f\colon M \to [0,1]$ such that $\ker\hat{f}\neq \set{0}$. Moreover $\lambda(M) = 0$, therefore $\F(M) \equiv \ell_1$ (see \cite{Godard}).}
\smallskip

\begin{proof}
Let $M=\set{0,1}\cup \set{x_n:n\in \N} \cup \set{y_n:n\in \N}$, where $x_{2^{n-1}+k}=2^{n-1}+k+1$ and $y_{2^{n-1}+k}=x_{2^{n-1}+k}+\frac{1}{4^{n}}$, $k=0,\ldots, 2^{n-1}-1$. Note that $d(x,y)\geq \frac{1}{2}$ for $x\neq y\in M$ except if $\{x,y\}=\{x_n,y_n\}$ for some $n$. Let $C \subset [0,1]$ be again the Smith-Volterra-Cantor set. To avoid confusion with the elements of $M$, we now write 
$(x_{2^{n-1}+k}',y_{2^{n-1}+k}')$, $k=0,\ldots, 2^{n-1}-1$, the intervals of length $4^{-n}$ which are removed at the $n$-th stage of the usual construction of a Cantor set.
We define $h\colon M \to [0,1]$ by $h(0)=0$, $h(1)=1$, $h(x_n)=x_n'$ and $h(y_n)=y_n'$.
Notice that $h \in \Lip_0(M,[0,1])$ and $h$ is injective.
Thanks to Example~\ref{exam:NotLipLin}, there exists $\mu' \in \Free([0,1])$ and an injective $f\in \Lip_0([0,1])$ such that $\hat{f}(\mu')$ = 0.
We have in fact $\mu'=\delta(1)-\sum_{n=1}^\infty (\delta(y_n')-\delta(x_n'))$.
Observe that $\mu'=\hat{h}(\mu)$ where $\mu=\delta(1)-\sum_{n=1}^\infty (\delta(y_n)-\delta(x_n))$.
The last claim is proved by noticing that the series defining $\mu$ converges absolutely (which is clear).
It only remains to prove that $\mu \neq 0$, which is done as follows :
\[ \sum_{n=1}^\infty \norm{\delta(y_n)-\delta(x_n)}= \sum_{n=1}^\infty \frac{2^{n-1}}{4^n} =\frac{1}{2}\]
and so $\norm{\mu}\geq 1/2$ as $\norm{\delta(1)}=1$.
\end{proof}
\end{example}

\begin{example}\label{e:CountableDiscrete}
\textit{There exists a bounded countable complete and discrete $M$ and a Lipschitz and injective $f\colon M \to [0,1]$ such that $\ker\hat{f}\neq \set{0}$.}
\smallskip

\begin{proof} 
Let $M=\set{0,1}\cup \set{x_n:n\in \N} \cup \set{y_n:n\in \N}$. 
We define all distances between distinct points to be $1$ except for $d(x_n,y_n)=4^{-k}$ when $n=2^{k-1},\ldots,2^k-1$. 
Now the rest of the proof is verbatim the same as the previous one, so we leave the details to the reader. 
\end{proof}
\end{example}

Notice that this example can be modified in an obvious manner to moreover obtain either 
\begin{itemize}
    \item[a)] $M$ is a subset of an $\R$-tree which contains all the branching points of that tree, i.e. $\Free(M)\equiv \ell_1$, or
    \item[b)] $M$ is ultrametric.
\end{itemize}

\section{The bidual}

In this section, we deal with operators of the kind $\hat{f}^{**}=C_f^*: \F(M)^{**} \to \F(N)^{**}$. 
Of course, if $\widehat{f}$ is non-injective, then so is $\hat{f}^{**}$. So, thanks to the previous section, assuming that $f$ is injective is clearly not sufficient for $\hat{f}^{**}$ to be injective. One goal of this section is to characterize when $\hat{f}^{**}$ is injective, for instance in terms of properties of $f$.

Recall that an operator $T\colon X \to Y$ is \emph{tauberian} if $T^{**-1}(Y) \subset X$. Let us give a few general facts about tauberian operators (we refer the reader to \cite{TauberianBook, KaltonWilansky} for more background information). 
It is clear that if $T$ is tauberian, then $\ker T^{**} \subset X$.
Therefore, for $T$ tauberian, $T$ is injective if and only if $T^{**}$ is injective.
Further, if $T$ has closed range, then $\ker T^{**} \subset X$ implies that $T$ is Tauberian (see \cite[page 251]{KaltonWilansky}).
For the above reasons it could be helpful to know under which condition on $f$, the operator $\hat{f}$ is tauberian.

\begin{proposition}\label{p:BiAdjointInjective}
Let $f \in \Lip_0(M,N)$ be an injective function. 
If $\ker\hat{f}^{**}\subset \Free(M)$ then $f$ is bi-Lipschitz.
In particular, if $\hat{f}$ is tauberian then $f$ is bi-Lipschitz.
\end{proposition}

\begin{proof}
Assume that $f$ is not bi-Lipschitz. 
Then there are sequences $(x_n)_n,(y_n)_n \subset M$ such that $x_n \neq y_n$ for every $n \in \N$ and moreover
$$\lim\limits_{n \to \infty} \frac{d(f(x_n),f(y_n))}{d(x_n,y_n)} = 0.$$
Let us denote
$$ \forall n \in \N, \quad m_{x_n y_n} := \frac{\delta(x_n) - \delta(y_n)}{d(x_n , y_n)}.$$
Notice that for every weak$^*$ accumulation point $\mu \in \F(M)^{**}$ of the sequence $(m_{x_ny_n}) \subset S_{\F(M)}$, we have
\[
\forall g\in \Lip_0(N), \quad \duality{\hat{f}^{**}\mu,g}=\duality{\mu,g\circ f}=0,
\]
that is $\mu \in \ker\hat{f}^{**}$.
Indeed, every weak$^*$ neighborhood 
\[V_\varepsilon=\set{\gamma \in \Free(M)^{**}:\abs{\duality{\gamma-\mu,g\circ f}}<\varepsilon}\] contains infinitely many terms $m_{x_ny_n}$, so the conclusion readily follows.
\smallskip

We will show that there is an accumulation point of $(m_{x_ny_n})$ in $(B_{\Free(M)^{**}},w^*)$ which does not belong to $\Free(M)$.
Assume that all accumulation points are in $\Free(M)$.
Assume first that all accumulation points are $0$.
Then $m_{x_ny_n}\to 0$ weakly, which is not possible by \cite[Corollary~2.10]{GLPPRZ}.
So there is some $\mu\neq 0$ in the weak closure of $(m_{x_ny_n})$.
By~\cite[Proposition~2.9]{GLPPRZ}, $\mu=m_{xy}$ for some $x\neq y\in M$.
Since $\hat{f}(\mu)=\hat{f}^{**}(\mu)=0$ we get that $f(x)=f(y)$, which contradicts the injectivity hypothesis on $f$.
\end{proof}

\begin{corollary}\label{c:BiAdjointInjective} Let $f\in \Lip_0(M,N)$.  
The following assertions are equivalent.
\begin{itemize}
    \item[(i)] $\hat{f}^{**}$ is injective.
    \item[(ii)] $f$ is injective and $\ker\hat{f}^{**}\subset \Free(M)$.
    \item[(iii)] $f$ is injective and $\hat{f}$ is tauberian.
    \item[(iv)] $f$ is bi-Lipschitz.
    \item[(v)] $\widehat{f}$ is injective with closed range.
    \item[(vi)] $C_f$ is onto.
    \item[(vii)] $C_f$ has dense range.
\end{itemize}
\end{corollary}
Notice that every point other than $(ii)$ and $(iii)$ implies implicitly that $f$ must be injective. Also, to the best of our knowledge, the equivalence between $(vi)$ and $(vii)$ seems to be new.  

\begin{proof}
The implications $(i)$ $\implies$ $(ii)$ and $(vi)$ $\implies$ $(vii)$ are trivial.
The implication $(ii)$ $\implies$ $(iv)$ follows from Proposition~\ref{p:BiAdjointInjective} while
$(iv)$ $\iff$ $(v)$ $\iff$ $(vi)$ follows from Proposition~\ref{fchapinjbilip}
(see also \cite[Proposition~2.25]{Weaver2}). Next,
the equivalence of $(i)$ and $(vi)$ follows from the general theory of adjoint operators (see \cite[Exercise~2.46]{FHHMZ} for instance). The above lines prove that all assertions, except for $(iii)$, are equivalent. Now notice that $(iii)$ $\implies$ $(ii)$ is obvious. Finally, since for a bi-Lipschitz map $f$, $\hat{f}$ has closed range (see Proposition~\ref{fchapinjbilip}), we obtain that $\hat{f}$ is tauberian whenever $\ker\hat{f}^{**} \subset \Free(M)$ (see \cite[page 251]{KaltonWilansky}), which proves that $(ii)$ and $(iv)$ $\implies$ $(iii)$.
\end{proof}

\begin{remark}
In~\cite{JNST2015} it is proved that there exist injective tauberian operators on $L_1[0,1]$ that have dense non-closed range.
The above corollary shows that such operators cannot be obtained as linearizations of Lipschitz maps $f\colon [0,1]\to [0,1]$ since, more generally, this combination of properties is excluded for linearization of Lipschitz maps between any two metric spaces. 
\end{remark}

We conclude the section with two examples.
\smallskip

\begin{example}
~
\begin{itemize}
    \item[a)] Let $M=N=[0,1]$ and $f(x)=x^2$. On the one hand  $\widehat{f}$ is injective thanks to Proposition~\ref{prop:biLipsupp}. On the other hand $\widehat{f}^{**}$ is not injective since $f$ is not bi-Lipschitz.
    \smallskip
    
    Let us prove by a direct argument that $\hat{f}^{**}$ is not injective.
    We will define $\mu \in \Free(M)^{**}$ such that $\widehat{f}^{**}(\mu)=0$. 
    First, we consider the subspace $E=\set{\varphi \in \Lip_0(M): \varphi'(0) \in \R}$ of $\Lip_0(M)$.
    We define $\mu(\varphi):=\varphi'(0)$ for every $\varphi \in E$.
    This is a bounded linear functional on $E$ and we extend it as a bounded linear functional on the whole $\Lip_0(M)$. 
    Now it is clear that 
    \[\forall \varphi \in \Lip_0(N), \quad \duality{\hat{f}^{**}\mu,\varphi}=\duality{\mu,\varphi\circ f}=0.\]
    So $\hat{f}^{**}$ is not injective.
    \smallskip
    
    \item[b)] Let $M=N=\set{\frac1n}\cup \set{0}$ and $f(x)=x^2$ (notice that $M$ is Lip-lin injective in this case). Then similarly one has $\widehat{f}$ is injective while $\widehat{f}^{**}$ is not injective. Now for an example of $\mu \in \ker(\widehat{f}^{**})$ as constructed above, one may consider the Hahn-Banach extension of $\varphi \mapsto \lim_n n\varphi(\frac1n)$. We leave the details to the reader.
\end{itemize}
\end{example}

\section{Final remarks and open questions}

Naturally, after dealing with the injectivity, one may wonder what is the situation with respect to surjectivity. That is, one can study the implications ``$f$ surjective $\implies$ $\widehat{f}$ surjective'' and ``$\widehat{f}$ surjective $\implies$ $f$ surjective''. In fact, none of these implications are true in general. 

To begin with, it is rather easy to find examples of surjective maps $f$ such that $\widehat{f}$ are not surjective. Indeed, in view of Proposition~\ref{fchapinjbilip}, whenever $\widehat{f}$ is injective but $f$ is not bi-Lipschitz, we obtain that $\widehat{f}$ cannot be surjective (otherwise $\widehat{f}$ would be a linear isomorphism, which can happen only when $f$ is bi-Lipschitz). For instance, we already explained that, for $f \colon x \in [0,1] \mapsto x^2 \in [0,1]$, $\widehat{f}$ is injective. However, since $f$ is not bi-Lipschitz, $\widehat{f}$ cannot be an isomorphism and so $\widehat{f}$ is not surjective.  

On the other hand, there are some situations where $\widehat{f}$ surjective implies $f$ surjective:
\begin{proposition}
Let $M,N$ be complete pointed metric spaces. If one of the following conditions is satisfied, then $f$ is surjective whenever $\widehat{f}$ is so.
\begin{enumerate}[$(a)$]
    \item $f$ is bi-Lipschitz;
    \item $\hat{f}$ is injective;
    \item $M$ is compact;
    \item $M$ is uniformly discrete and bounded.
\end{enumerate}
\end{proposition}
\begin{proof}
It is rather easy to see that $f$ has dense range if and only if $\widehat{f}$ has dense range (see e.g. \cite[Proposition~2.1]{ACP20}). So, if the range of $f$ is closed and $\widehat{f}$ is surjective, then $f$ must be surjective. This implies assertions (a) and (c). Assertion (b) follows from the fact that if $\widehat{f}$ is both injective and surjective, then it is an isomorphism, and therefore $f$ must be bi-Lipschitz. Finally, let us prove (d): given $y\in N$, let $\mu=\sum_{n\geq 0}a_n \delta(x_n)\in \mathcal F(M)$ with $\hat{f}(\mu)=\delta(y)$. Then $\sum_{n\geq 0} a_n \delta(f(x_n)) = \delta(y)$, where $(a_n)\in \ell_1$. By Lemma~\ref{l:ACseries}, there exists $n$ such that $y=f(x_n)$, so $y\in f(M)$. 
\end{proof}

The next example witnesses the fact that there are Lipschitz maps such that $\widehat{f}$ is surjective while $f$ is not surjective.

\begin{example}\label{ex:fHatSfNonS}
There are complete separable pointed metric spaces $M, N$ for which there is $f\in \Lip_0(M,N)$ such that $\widehat{f}$ is surjective but $f$ is not surjective. In the constructions below, both of the spaces are subsets of $\R$-trees and both free spaces are isometric to $\ell_1$. On the other hand, $f$ will not be injective. 

Let $N=\set{y_n\in \R:y_n=1-\frac1n, n \in \N \cup \set{\infty}}$ (here $y_\infty=1$) together with the induced distance from $\R$. (Notice that $y_1=0$.)
Let $M=\set{x_n:n \in \N} \cup \set{x_n': n\in \N}$ where $0=x_1$ is the base point.
We consider $M$ as a subspace of an $\R$-tree with the only branching point at $0$ from which it stems an infinity of branches (i.e. isometric copies of $[0,\infty)$) $b_n$. 
For every $n \in \N$ we have $x_n', x_{n+1} \in b_n$ in such a way that $d(0,x_n')=1$ and $d(0,x_{n+1})=1+d(y_{n+1},y_n)$.
Further we define $f:M \to N$ as 
$f(x_n)=f(x_n')=y_n$ for every $n \in \N$ (and $f(0)=0$).
One can check easily that $f$ is Lipschitz.
We also see immediately that $y_\infty \notin f(M)$.
On the other hand, it follows from Godard's work \cite{Godard} that $\Free(N)$ is isometric to $\ell_1$,
with $(m_{y_{n+1}y_n})_{n \in \N}$ being the $\ell_1$-basis isometrically, where $m_{y_{n+1}y_n} = d(y_{n+1},y_n)^{-1}(\delta(y_{n+1})-\delta(y_n))$. 
Now clearly, $\widehat{f}\left(m_{x_{n+1}x_n'}\right)=m_{y_{n+1}y_n}$ for every $n \in \N$.
Thus $\widehat{f}(\overline{\lspan}(m_{x_{n+1}x_n'}))=\Free(N)$.
\end{example}

We may interpret such an example in a more abstract setting as follows.

\begin{example}
Let $M$ be a metric space, $\sim$ an equivalence relation on $M$ and $M/\approx$ the metric quotient defined as in \cite[Definition~1.22]{Weaver2}, and let $M_\sim$ be the completion of $M/\approx$. Consider the canonical projection $f\colon M\to M_\sim$ sending each element to its equivalence class. Then $C_f\colon \Lip_0(M_\sim)\to \Lip_0(M)$ is an isometry (see Proposition~2.28 in \cite{Weaver2}). Thus, $\hat{f}$ is surjective. Now, if $M$ and $\sim$ are chosen so that $M/\approx$ is not complete (this is the case, for instance, of Example~1.24 in \cite{Weaver2}), then $f(M)=M/\approx$ is a proper subset of $M_\sim$. So, $f$ is not surjective.
\end{example}

We will now conclude the paper with some open questions. Recall that we proved in Corollary~\ref{cor:injEQps} that, when $M$ is bounded, a Lipschitz map $f : M \to N$ preserves supports if and only if $\widehat{f} : \F(M) \to \F(N)$ is injective. An obvious question is whether this result remains valid for a general metric space $M$. Since one implication is always true, see Proposition~\ref{PSimpliesINJ}, it only remains one implication to study:
\begin{question}
Let $M$ be an unbounded metric space. Assume that $f : M \to N$ is a Lipschitz map such that $\widehat{f} : \F(M) \to \F(N)$ is injective. Is it true that $f$ preserves supports? 
\end{question}

In Theorem~\ref{t:CharCptTotDiscoLLI} 
we proved, in the compact setting, that $M$ is Lip-lin injective and totally disconnected if and only if $s_0(M)$  
uniformly separates the points of $M$.
Unfortunately, we do not have any characterisation without disconnectedness assumption. In fact, we know that being compact and totally disconnected is not sufficient to be Lip-lin injective (Proposition~\ref{p:p1uNotSufficient}), but it is not clear whether every Lip-lin injective space must be totally disconnected. On the other hand, we proved in 
Corollary~\ref{c:LipLinInjectiveImpliesp1u} that it must be purely 1-unrectifiable.
\begin{question}
Find a (metric) characterisation of compact Lip-lin injective metric spaces. Are they always totally disconnected? Or totally path-disconnected? 
\end{question}
Of course, the same question in the general (non-compact) case is left open as well.

\section*{Acknowledgments}

The first author was supported by the grants of Ministerio de Economía, Industria y Competitividad
MTM2017-83262-C2-2-P and PID2021-122126NB-C32; and Fundación Séneca Región de Murcia 20906/PI/18. The two last authors were partially supported by the French ANR project No. ANR-20-CE40-0006. 
A part of this work was written while the last named author was visiting Departamento de Matemáticas of Universidad de Zaragoza in spring 2021. He wishes to express his gratitude for their hospitality and excellent working conditions during the stay. We thank Fundación CAI-Ibercaja for partial support.
We also thank Chris Gartland, Pedro Tradacete and Abraham Rueda Zoca for useful comments. Finally, we are very grateful to the referees for their careful reading of the paper and their useful suggestions.

\end{document}